\documentclass[a4]{amsart}

\usepackage{amssymb}
\usepackage[all]{xy}
\usepackage{amsmath, amsfonts, amsthm , amssymb}

\input xypic 
\usepackage{tikz} 

\newtheorem{theorem}{Theorem}[section]
\newtheorem{proposition}[theorem]{Proposition}
\newtheorem{lemma}[theorem]{Lemma}
\newtheorem{corollary}[theorem]{Corollary}

\theoremstyle{definition}
\newtheorem{definition}[theorem]{Definition}
\newtheorem{example}[theorem]{Example}
\newtheorem{remark}[theorem]{Remark}

\newcommand{\lra}{\longrightarrow}

\newcommand{\CP}{\ensuremath{\mathbb{C}P^{\infty}}}

\newcommand{\dr}[3]{\ensuremath{#1\stackrel{#2}
{\longrightarrow}#3}}

\newcommand{\cpinf}{\ensuremath{\mathbb{C}P^{\infty}}}
\newcommand{\djk}{\ensuremath{DJ_{K}}}
\newcommand{\zk}{\ensuremath{\mathcal{Z}_{K}}}
\newcommand{\djks}{\ensuremath{DJ_{K}(\underline{S})}}
\newcommand{\zks}{\ensuremath{\mathcal{Z}_{K}(\underline{S})}}
\newcommand{\djkx}{\ensuremath{DJ_{K}(\underline{X})}}
\newcommand{\zkx}{\ensuremath{\mathcal{Z}_{K}(\underline{X})}}
\newcommand{\djkstwo}{\ensuremath{DJ_{K}(S^{2})}}
\newcommand{\zkstwo}{\ensuremath{\mathcal{Z}_{K}(\mathcal{S}^{2})}}
\newcommand{\djky}{\ensuremath{DJ_{K}(\underline{\Sigma Y})}}
\newcommand{\zky}{\ensuremath{\mathcal{Z}_{K}(\underline{\Sigma Y})}}
\newcommand{\colim}{\ensuremath{\mbox{colim}}}

\newcommand{\hlgy}[1]{\ensuremath{H_{*}(#1)}}

\newcommand{\qhlgy}[1]{\ensuremath{H_{*}(#1;\mathbb{Q})}}
\newcommand{\rqhlgy}[1]{\ensuremath{\widetilde{H}_{*}(#1;\mathbb{Q})}}
\newcommand{\cohlgy}[1]{\ensuremath{H^{*}(#1)}}

\newcounter{bean}
\newenvironment{letterlist}{\begin{list}{\rm ({\alph{bean}})}
      {\usecounter{bean}\setlength{\rightmargin}{\leftmargin}}}
      {\end{list}}

\newcommand{\namedright}[3]{\ensuremath{#1\stackrel{#2}
 {\longrightarrow}#3}}
\newcommand{\nameddright}[5]{\ensuremath{#1\stackrel{#2}
 {\longrightarrow}#3\stackrel{#4}{\longrightarrow}#5}}
\newcommand{\namedddright}[7]{\ensuremath{#1\stackrel{#2}
 {\longrightarrow}#3\stackrel{#4}{\longrightarrow}#5
  \stackrel{#6}{\longrightarrow}#7}}

\newcommand{\larrow}{\relbar\!\!\relbar\!\!\rightarrow}
\newcommand{\llarrow}{\relbar\!\!\relbar\!\!\larrow}
\newcommand{\lllarrow}{\relbar\!\!\relbar\!\!\llarrow}

\newcommand{\llnamedright}[3]{\ensuremath{#1\stackrel{#2}
 {\llarrow}#3}}

\newcommand{\lllnamedright}[3]{\ensuremath{#1\stackrel{#2}
 {\lllarrow}#3}}

\newcommand{\lllnamedddright}[7]{\ensuremath{#1\stackrel{#2}
 {\lllarrow}#3\stackrel{#4}{\lllarrow}#5
  \stackrel{#6}{\lllarrow}#7}}

\newcommand{\qqed}{\hfill\Box}

\begin{document}
\title{Higher Whitehead products in toric topology}
\author{Jelena Grbi\'{c}}
\address{School of Mathematics, University of Manchester,
         Manchester M13 9PL, United Kingdom} 
\address{\textit{Current address}: School of Mathematics, University of Southampton, 
         Southampton SO17 1BJ} 
\email{J.Grbic@soton.ac.uk}
\author{Stephen Theriault}
\address{Institute of Mathematics,
         University of Aberdeen, Aberdeen AB24 3UE, United Kingdom} 
\address{\textit{Current address}: School of Mathematics, University of Southampton, 
         Southampton SO17 1BJ} 
\email{S.D.Theriault@soton.ac.uk}

\date{}

\begin{abstract} 
In toric topology, to each simplicial complex $K$ on $n$ vertices one 
associates two key spaces, the Davis-Januskiewicz space $DJ_{K}$ and 
the moment-angle complex $\zk$, which are related by a homotopy fibration 
\(\nameddright{\zk}{\widetilde{w}}{DJ_{K}}{}{\prod_{i=1}^{n}\cpinf}\). 
A great deal of work has been done to study properties of the two spaces 
$DJ_{K}$ and $\zk$. However, very little is known about the map $\widetilde{w}$. 
In this paper we show that, for a certain family of simplicial complexes $K$,  
the map $\widetilde{w}$ is a sum of higher and iterated Whitehead products. 
\end{abstract} 

\maketitle

\section{Introduction} 
To a simplicial complex $K$ on $n$ vertices, Davis and
Januszkiewicz~\cite{DJ} associated two fundamental objects of toric
topology: the moment-angle complex $\zk$ and the Davis-Januszkiewicz
space~$\djk$, whose study connects algebraic geometry, topology,
combinatorics, and commutative algebra. Algebraic topologists, on
their side, have tried to understand the homotopy theory of these
spaces and other related topological spaces with a torus action.
Recent developments~\cite{BP} have shown that from the homotopy
theoretical point of view both spaces $\zk$ and $\djk$ could be
considered as polyhedral products of the topological pairs
$(D^2,S^1)$ and $(\CP,*)$, respectively. That lead to a wide
generalisation~\cite{BP2,S} to polyhedral products
$(\underline{X,A})^K$ of $n$-tuples of topological pairs
$(\underline{X,A})=(X_i,A_i)_{i=1}^{n}$. These spaces are also very 
closely related to combinatorics and commutative algebra, and for 
particular topological pairs, to algebraic geometry and convex geometry. 

To date, there has been considerable success in studying the
homotopy type of $\zk$~\cite{GT1,GT2,IK} or its suspension~\cite{BBCG}.
However, no attempt has been made to study the map
\(\namedright{\zk}{}{\djk}\). This means that the interesting
information that is known about the moment-angle complex \zk\ cannot
be related to the Davis-Januszkiewicz space \djk. The purpose of
this paper is to remedy this deficiency. We show that for a certain
family of simplicial complexes $K$, \zk\ is homotopy equivalent to a
wedge of spheres and the homotopy equivalence may be chosen so that
the map \(\namedright{\zk}{}{\djk}\) consists of a specified
collection of higher Whitehead products and iterated Whitehead
products. In particular, each missing face of $K$ corresponds to
the existence of a nontrivial higher Whitehead product whose adjoint
has a nonzero Hurewicz image in $\qhlgy{\Omega\djk}$. 

Let $X_{1},\ldots,X_{n}$ be path-connected spaces and let
$\underline{X}=\{X_{1},\ldots,X_{n}\}$. Following~\cite{BP2, BBCG},
for $\sigma=(i_1,\ldots,i_k)$ let $X^\sigma=\prod_{j=1}^{k}
X_{i_j}$, and let $\djkx=\bigcup_{\sigma\in K}X^\sigma$. Notice that
there is an inclusion $\dr{\djk(\underline{X})}{}{\prod_{i=1}^{n}
X_{i}}$. Define $\zkx$ by the homotopy fibration
\begin{equation}
  \label{DJfib}
  \nameddright{\zk(\underline{X})}{}{\djk(\underline{X})}{}
    {\prod_{i=1}^{n} X_{i}}.
\end{equation}
If $X_{1},\ldots,X_{n}$ all equal a common space $X$, we instead write
$\zk(X)$ and $\djk(X)$. When $X_{i}=\cpinf$ for each $1\leq i\leq n$, 
the homotopy fibration~(\ref{DJfib}) specializes to the case of
primary interest in toric topology, that is, to the homotopy
fibration $\nameddright{\zk}{}{\djk}{}{\prod_{i=1}^{n}\cpinf}$. For
the sake of clarity, we remark that in the terms of polyhedral
products our $\zk(\underline{X})$ is actually 
$(\underline{\mathrm{Cone}\,\Omega X,\Omega X})^K$, 
whereas $\djk(\underline{X})=(\underline{X,*})^K$. 

These polyhedral products currently enjoy great popularity;
in particular their loop homology with various coefficients and for
different families of simplicial complexes has been calculated. Some
simple but important examples of the homology of $\Omega\djkx$ were
calculated by Lemaire~\cite{L} in 1974 before the notion of $\zkx$
and $\djkx$ were introduced. Panov and Ray~\cite{PR} introduced
categorical formalism to study the loop homology of \djk\ and gave
explicit calculations when $K$ is a flag complex.
Dobrinskaya~\cite{D} has a general approach for calculating the
homology of $\Omega\djk(\underline{X})$ for an arbitrary simplicial
complex $K$ in terms of the homology of $\Omega(\underline{X})$ and
some special relations coming from the homology of $\Omega\djk(S^2)$.  
However, the homology of $\Omega\zk(\underline{X})$ remains a mystery. 

In this paper we first consider the case when each $X_{i}$ is a
sphere, writing $\underline{S}=(S^{m_{1}+1},\ldots,S^{m_{n}+1})$. 
As an intermediate goal towards understanding the map
$\zk\lra\djk$ we need to calculate the rational homology of
$\Omega\djks$ and $\Omega\djk$. However, it is important to
emphasize that we do this in such a way as to remember the geometry
of the space, that is, in such a way as to keep track of specific
Hurewicz images. The existing models for rational loop homology are
not known to do this, so we have to produce our own model which
does. The methods we use lend themselves well to concrete calculations, 
and we include some examples to illustrate this. 

In what follows $K$ will be a simplicial complex on $n$ vertices, 
labelled $\{1,\ldots,n\}$. That is, a simplex $\sigma\in K$ corresponds to 
a sequence $(i_{1},\ldots,i_{k})$ where $1\leq i_{1}<\cdots<i_{k}\leq n$ and
the integers $i_{j}$ are the vertices of $K$ which are in $\sigma$.
Let $\vert\sigma\vert=k-1$ be the dimension of $\sigma$. We
concentrate on the collection $MF(K)$ of \emph{missing faces}. To be
precise, a sequence $(i_{1},\ldots,i_{k})$ is in $MF(K)$ if: (i)
$(i_{1},\ldots,i_{k})\notin K$, and (ii) every proper subsequence of
$(i_{1},\ldots,i_{k})$ is in $K$. For example, let $K$ be the
simplicial complex on $4$ vertices 
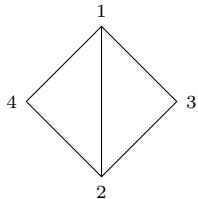
\begin{equation} 
  \label{4vexample} 
  \begin{tikzpicture} 
     \draw (0,0)--(1,1)--(2,0)--(1,-1)--(0,0); 
     \draw (1,1)--(1,-1); 
     \node [above] at (1,1) {\scriptsize 1}; 
     \node [right] at (2,0) {\scriptsize 3}; 
     \node [below] at (1,-1) {\scriptsize 2}; 
     \node [left] at (0,0) {\scriptsize 4}; 
  \end{tikzpicture} 
\end{equation} 
Then $MF(K)=\{(3,4),(1,2,3),(1,2,4)\}$.

\begin{definition}
Let $K$ be a simplicial complex on $n$ vertices. We say that $K$
is an \emph{$MF$-complex} if
\[\vert K\vert=\bigcup_{\sigma\in MF(K)}\vert\partial\sigma\vert\]
where $\vert K\vert$ and $\vert\partial\sigma\vert$ denote the
geometrical realisations of $K$ and $\partial\sigma$, respectively.
\end{definition} 

A simple example of a non-$MF$-complex is the boundary of a square. 

The class of $MF$-complexes is larger than we want for producing 
wedge decompositions of \zks\ and \zk. We will give an example in 
Section~\ref{sec:objects} of an $MF$-complex $K$ 
with the property that $\zk$ has nontrivial cup-products. We wish 
to avoid this, so we add another condition which restricts how the 
faces of $K$ can be assembled. 

\begin{definition} 
Let $K$ be an $MF$-complex on $n$ vertices. We say that $K$ is a  
\emph{directed $MF$-complex} if there is a sequence of subcomplexes 
$\emptyset\subseteq K_{1}\subseteq\cdots\subseteq K_{l}=K$ 
for some $l$, where $K_{i}=K_{i-1}\cup\partial\sigma_{i}$ for 
$\sigma_{i}\in MF(K)$ and $K_{i-1}\cap\sigma_{i}$ is a face common 
to $K_{i-1}$ and $\sigma_{i}$. 
\end{definition} 

Observe that the simplicial complex~(\ref{4vexample}) is a directed 
$MF$-complex, but if the edge $(3,4)$ is also added (giving the 
$1$-skeleton of a tetrahedron) then the resulting simplicial complex 
is an $MF$-complex but not a directed $MF$-complex. More 
examples are given in Section~\ref{sec:objects}. 

Directed $MF$-complexes have the property that \zk\ decomposes 
as a wedge of spheres. In particular, all cup products and higher Massey 
products in $\cohlgy{\zk}$ are zero. We obtain this as a special case of a 
property for more general spaces. 

\begin{theorem}
   \label{Mcomplexwedge}
   Let $K$ be a directed $MF$-complex on $n$ vertices. Let
   $\underline{X}=\{X_{1},\ldots,X_{n}\}$ where each $X_{i}$ is
   a path-connected topological space. Then $\zk(\underline{X})$ is
   homotopy equivalent to a wedge of spaces of the form
   $\Sigma^{t}\Omega X_{i_{1}}\wedge\cdots\wedge\Omega X_{i_{k}}$
   for various $1\leq t<n$ and sequences $(i_{1},\ldots,i_{k})$
   where ${1\leq i_{1}<\cdots<i_{k}\leq n}$.
\end{theorem}

\begin{corollary}
   \label{integralM}
   Let $K$ be a directed $MF$-complex on $n$ vertices. Then each of
   \zks\ and \zk\ is homotopy equivalent to a wedge of
   simply-connected spheres.
\end{corollary}

Next, as an intermediate step, we calculate $\qhlgy{\Omega\djks}$
and $\qhlgy{\Omega\djk}$ using an Adams-Hilton model, with the
emphasis on detecting Hurewicz images. To state this we introduce
some notation. If $V$ is a graded $\mathbb{Q}$-vector space, let
$L\langle V\rangle$ be the free graded Lie algebra generated by~$V$,
and let $UL\langle V\rangle$ be its universal enveloping algebra. If $V$ 
has basis $\{v_{1},\ldots,v_{n}\}$ let $L_{ds}\langle v_{1},\ldots,v_{n}\rangle$ 
be the direct sum $\oplus_{i=1}^{n} L\langle v_{i}\rangle$. In particular, 
in $L_{ds}\langle v_{1},\ldots,v_{n}\rangle$ we have $[v_{i},v_{j}]=0$ 
if $i\neq j$. Notice that if $v_{i}$ is of even degree then $[v_{i},v_{i}]=0$ 
but this is not the case if $v_{i}$ is of even degree. On the other hand, 
if $v_{i}$ is of even degree then in $UL\langle v_{i}\rangle$ we have 
$[v_{i},v_{i}]=2v_{i}^{2}$, so for $v_{i}$ of any parity we have 
$UL\langle v_{i}\rangle\cong\mathbb{Q}[v_{i}]$. Thus 
$UL_{ds}\langle v_{1},\ldots,v_{n}\rangle\cong\otimes_{i=1}^{n} UL\langle v_{i}\rangle$. 
If $L$ is a Lie algebra and $x_{1},\ldots,x_{k}\in L$, let
$[[x_{1},x_{2}],\ldots,x_{k}]$ denote the iterated bracket
$[\ldots[[x_{1},x_{2}],x_{3}],\ldots,x_{k}]$.

Let $b_{i}$ be the Hurewicz image of the adjoint of the coordinate
inclusion \(\namedright{S^{m_{i}+1}}{}{\djks}\). Abusing notation,
let $b_{i}$ also be the Hurewicz image of the adjoint of the
composite \(\nameddright{S^{2}}{}{\cpinf}{}{\djk}\), where the left
map is the inclusion of the bottom cell and the right map is the
inclusion of the $i^{th}$-coordinate. By $u_\sigma$ we denote the
Hurewicz image of the adjoint of the Whitehead product corresponding
to a missing face $\sigma\in MF(K)$. We will phrase
$\qhlgy{\Omega\djks}$ and $\qhlgy{\Omega\djk}$ as quotients of
$U(L_{ds}\langle b_{1},\ldots,b_{n}\rangle\textstyle\coprod
   L\langle u_{\sigma}\mid\sigma\in MF(K)\rangle)$.
A distinction needs to be made between the elements~$u_{\sigma}$
where $\vert\sigma\vert=1$ and $\vert\sigma\vert>1$. The latter
elements are independent from $b_{1},\ldots,b_{n}$. On the other
hand, if $\vert\sigma\vert=1$ then $\sigma=(i_{1},i_{2})$ and
$u_{\sigma}=[b_{i_{1}},b_{i_{2}}]$, which is not independent from
$b_{1},\ldots,b_{n}$. This leads to additional relations determined
by the graded Jacobi identity and face relations. Specifically, we
have $[u_{\sigma},b_{j}]=[[b_{i_{1}},b_{i_{2}}],b_{j}]=
    [b_{i_{1}},[b_{i_{2}},b_{j}]]-(-1)^{\vert b_{i_{1}}\vert\vert b_{i_{2}}\vert}
    [b_{i_{2}},[b_{i_{1}},b_{j}]]$
and if $(i_{1},j)\in K$ or $(i_{2},j)\in K$ then $[b_{i_{1}},b_{j}]=0$ 
or $[b_{i_{2}},b_{j}]=0$ repsectively. The collection of such
relations forms an ideal in $U(L_{ds}\langle
b_{1},\ldots,b_{n}\rangle\textstyle\coprod
   L\langle u_{\sigma}\mid\sigma\in MF(K)\rangle)$
which we label as $J$. Note that if every missing face $\sigma\in MF(K)$
is of dimension greater than $1$, then $J$ is trivial.

\begin{theorem}
   \label{ULcolim}
   Let $K$ be a directed $MF$-complex. There is an algebra isomorphism
   \[\qhlgy{\Omega\djks}\cong U(L_{ds}\langle b_{1},\ldots,b_{n}\rangle
         \textstyle\coprod L\langle u_{\sigma}\mid\sigma\in MF(K)\rangle)/J\]
   where each $u_{\sigma}$ is the Hurewicz image of the adjoint of a
   higher Whitehead product. Further, the loop map 
   \(\namedright{\Omega\djks}{}{\prod_{i=1}^{n}\Omega S^{m_{i}+1}}\)
   is modelled by the map
   \[\namedright{U(L_{ds}\langle b_{1},\ldots,b_{n}\rangle\textstyle\coprod
        L\langle u_{\sigma}\mid\sigma\in MF(K)\rangle)/J}
       {U(\pi)}{UL_{ds}\langle b_{1},\ldots,b_{n}\rangle}\]
   where $\pi$ is the projection.
\end{theorem}

Let
\(\imath\colon\namedright{S^{2}}{}{\cpinf}\) 
be the inclusion of the bottom cell. For any simplicial complex $K$, 
this induces a map 
\(DJ_{k}(\iota)\colon\namedright{\djk(S^{2})}{}{\djk}\). 

\begin{theorem}
   \label{hlgyloopdjk}
   Let $K$ be a directed $MF$-complex. There is an algebra isomorphism
   \[\qhlgy{\Omega\djk}\cong
        U(L_{ds}\langle b_{1},\ldots,b_{n}\rangle\textstyle\coprod
        L\langle u_{\sigma}\mid\sigma\in MF(K)\rangle)/(I+J)\]
   where $u_{\sigma}$ is the Hurewicz image of the adjoint of a
   higher Whitehead product and $I$ is the ideal
   \[I=(b_{i}^{2},[u_{\sigma},b_{j_{\sigma}}]\mid
       1\leq i\leq n,\sigma=(i_{1},\ldots,i_{k})\in MF(K),
       j_{\sigma}\in\{i_{1},\ldots,i_{k}\}).\]
   Further, there is a commutative diagram of algebras
   \[\diagram
        \qhlgy{\Omega\djkstwo}\rto^-{\cong}\dto^{(\Omega\djk(\imath))_{\ast}}
            & U(L_{ds}\langle b_{1},\ldots,b_{n}\rangle\textstyle\coprod
                L\langle u_{\sigma}\mid\sigma\in MF(K)\rangle)/J\dto^{q} \\
        \qhlgy{\Omega\djk}\rto^-{\cong}
            & U(L_{ds}\langle b_{1},\ldots,b_{n}\rangle\textstyle\coprod
                L\langle u_{\sigma}\mid\sigma\in MF(K)\rangle)/(I+J)
     \enddiagram\]
   where $q$ is the quotient map.
\end{theorem} 

Notice that the relation $b_{i}^{2}$ in $I$ lets us replace 
$L_{ds}\langle b_{1},\ldots,b_{n}\rangle$ in the statement of 
Theorem~\ref{hlgyloopdjk} by $L_{ab}\langle b_{1},\ldots,b_{n}\rangle$, 
where $L_{ab}$ is the free abelian Lie algebra generated by the indicated 
elements, which is characterized by having its bracket identically zero. 

Our main theorems are homotopy theoretic. For $1\leq i\leq n$, let
\(a_{i}\colon\namedright{S^{m_i+1}}{}{\djk(\underline{S})}\) be the
inclusion of the $i^{th}$-coordinate. The
analogue of the algebraic ideal $J$ occurs when
$\sigma=(i_{1},i_{2})$, in which case there is a Whitehead product
$w_{\sigma}=[a_{i_{1}},a_{i_{2}}]$; as in the algebraic case, this 
imposes relations determined by the Jacobi identity and face
relations in cases of the form $[w_{\sigma},a_{j}]$ when
$(i_{1},j)\in K$ or $(i_{2},j)\in K$. For $\vert\sigma\vert=1$, let
$W_{\sigma}$ be the collection of all independent Whitehead products
of the form  $[[w_{\sigma},a_{j_{1}}]\ldots, a_{j_{l}}]$ where
$1\leq j_{1}\leq\cdots\leq j_{l}\leq l$ and $1\leq l<\infty$.

\begin{theorem}
   \label{djksmaps}
   Let $K$ be a directed $MF$-complex on $n$ vertices, so that there is a
   homotopy equivalence $\zks\simeq\bigvee_{\alpha\in\mathcal{I}} S^{t_{\alpha}}$.
   The equivalence can be chosen so that the composite
   \[\nameddright{\bigvee_{\alpha\in\mathcal{I}} S^{t_{\alpha}}}{}{\zks}
        {}{\djks}\]
   is a wedge sum of the following maps:
   \begin{letterlist}
      \item a higher Whitehead product
            \(w_{\sigma}\colon\namedright{S^{t_{\sigma}}}{}{\djks}\)
            for each missing face $\sigma=(i_{1},\ldots,i_{k})\in MF(K)$, where
            $t_{\sigma}=k-1+(\Sigma_{j=1}^{k} m_{i_{j}})$;
      \item an iterated Whitehead product
            \[[[w_{\sigma},a_{j_{1}}]\ldots, a_{j_{l}}]\colon
                  \namedright{S^{t_{\alpha}}}{}{\djks}\]
            for each $\sigma\in MF(K)$ of dimension greater than $1$ and
            each list $1\leq j_{1}\leq\cdots\leq j_{l}\leq l$, where
            $1\leq l<\infty$ and $t_{\alpha}=t_{\sigma}+\Sigma_{t=1}^{l} m_{j_{t}}$;
      \item the collection of independent iterated Whitehead products
            $W_{\sigma}$ for each $\sigma\in MF(K)$ of dimension $1$.
   \end{letterlist}
\end{theorem} 

Let $\tilde{a}_{i}$ be the composite
\(\tilde{a}_{i}\colon\nameddright{S^{2}}{\imath}{\cpinf}{}{\djk}\)
where the right map is the inclusion of the $i^{th}$-coordinate.  
If $\sigma=(i_{1},i_{2})$, let $\widetilde{w}_{\sigma}$ be the
Whitehead product $[\tilde{a}_{i_{1}},\tilde{a}_{i_{2}}]$. As above,
for $\vert\sigma\vert=1$, let $\widetilde{W}_{\sigma}$ be the
collection of all independent Whitehead products of the form
$[[\widetilde{w}_{\sigma},\tilde{a}_{j_{1}}]\ldots,
\tilde{a}_{j_{l}}]$ where $1\leq j_{1}\leq\cdots\leq j_{l}\leq l$
and $1\leq l<\infty$. Given $\sigma=(i_{1},\ldots,i_{k})$, let
$J_{\sigma}=\{1,\ldots,n\}-\{i_{1},\ldots,i_{k}\}$.

\begin{theorem}
   \label{djkmaps}
   Let $K$ be a directed $MF$-complex on $n$ vertices, so that there is a
   homotopy equivalence
   $\zk\simeq\bigvee_{\widetilde{\alpha}\in\mathcal{\widetilde{I}}}
        S^{t_{\widetilde{\alpha}}}$.
   The equivalence can be chosen so that the composite
   \[\nameddright{\bigvee_{\widetilde{\alpha}\in\mathcal{\widetilde{I}}}
        S^{t_{\widetilde{\alpha}}}}{}{\zk}{}{\djk}\]
   is a wedge sum of the following maps:
   \begin{letterlist}
      \item a higher Whitehead product
            \(\widetilde{w}_{\sigma}\colon\namedright{S^{2\vert\sigma\vert+1}}{}{\djk}\)
            for each missing face $\sigma\in MF(K)$;
      \item an iterated Whitehead product
            \[[[\widetilde{w}_{\sigma},\tilde{a}_{j_{1}}]\ldots, \tilde{a}_{j_{l}}]\colon
                  \namedright{S^{2\vert\sigma\vert+l+1}}{}{\djk}\]
            for each $\sigma\in MF(K)$ of dimension greater than $1$ and each list
            $j_{1}<\cdots<j_{l}$ in $J_{\sigma}$, where $1\leq l\leq n$;
      \item the collection of independent iterated Whitehead products
            $\widetilde{W}_{\sigma}$ for each $\sigma\in MF(K)$ of dimension $1$.
   \end{letterlist}
\end{theorem}

Although in this paper our goal is to identify the map $\zk\lra\djk$
for $K$ a directed $MF$-complex, we expect this can be generalized 
first to a much larger family of simplicial complexes and second to the map
between polyhedral products of any $n$ tuple of topological pairs.
As a consequence one would obtain a homotopy wedge decomposition of
$\zk(\underline{X,A})$, which will reduce to a wedge of spheres in
the case of $\zk$.

\section{The objects of study}
\label{sec:objects}

This section gives an initial analysis of directed $MF$-complexes. First,
we compare directed $MF$-complexes to another family of simplicial complexes
that has received considerable attention for its role in producing
wedge decompositions of \zk. Then we prove the wedge decompositions 
in Theorem~\ref{Mcomplexwedge} and Corollary~\ref{integralM}.

A simplicial complex $K$ on $n$ vertices is \emph{shifted} if there 
is an ordering on the
vertex set such that whenever $\sigma$ is a simplex of $K$ and
$v^{\prime}<v$, then $(\sigma - v)\cup v^{\prime}$ is a simplex of
$K$. In~\cite{GT1} it was shown that if $K$ is a shifted complex then~\zk\ 
is homotopy equivalent to a wedge of spheres. In fact, in~\cite{GT1} it was 
shown that if $K$ is shifted then the polyhedral product 
$(\underline{C\Omega X,\Omega X})$ is homotopy equivalent 
to a wedge of suspensions, where $C\Omega X$ is the cone on $\Omega X$. 
This was generalized to any polyhedral product 
$(\underline{CX,X})^{K}$ in~\cite{GT2,IK}. 

We show that directed $MF$-complexes and shifted complexes form distinct
families, with nontrivial intersection. Consider the three examples: 
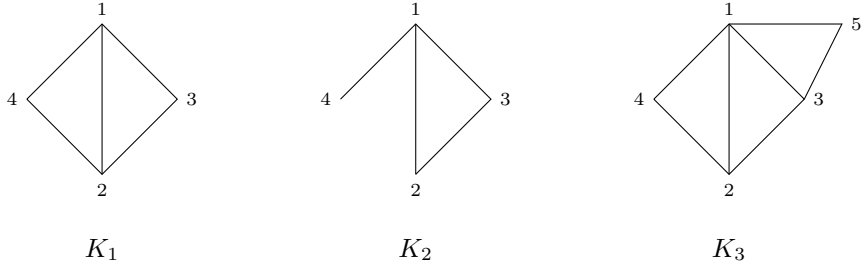
\[\begin{tikzpicture} 
     \draw (0,0)--(1,1)--(2,0)--(1,-1)--(0,0); 
     \draw (1,1)--(1,-1); 
    \node [above] at (1,1) {\scriptsize 1}; 
     \node [right] at (2,0) {\scriptsize 3}; 
     \node [below] at (1,-1) {\scriptsize 2}; 
     \node [left] at (0,0) {\scriptsize 4}; 
     \node at (1,-2) {$K_{1}$}; 
  \end{tikzpicture}\qquad\qquad  
  \begin{tikzpicture} 
     \draw (0,0)--(1,1)--(2,0)--(1,-1); 
     \draw (1,1)--(1,-1); 
     \node [above] at (1,1) {\scriptsize 1}; 
     \node [right] at (2,0) {\scriptsize 3}; 
     \node [below] at (1,-1) {\scriptsize 2}; 
     \node [left] at (0,0) {\scriptsize 4}; 
     \node at (1,-2) {$K_{2}$}; 
  \end{tikzpicture} \qquad\qquad  
  \begin{tikzpicture} 
     \draw (0,0)--(1,1)--(2,0)--(1,-1)--(0,0); 
     \draw (1,1)--(1,-1); 
     \draw (1,1)--(2.5,1)--(2,0); 
     \node [above] at (1,1) {\scriptsize 1}; 
     \node [right] at (2,0) {\scriptsize 3}; 
     \node [below] at (1,-1) {\scriptsize 2}; 
     \node [left] at (0,0) {\scriptsize 4}; 
     \node [right] at (2.5,1) {\scriptsize 5}; 
     \node at (1,-2) {$K_{3}$}; 
  \end{tikzpicture}\]  
Observe that $K_{1}$ and $K_{2}$ are shifted but $K_{3}$ is not. The
list of minimal missing faces in each case is:
\begin{enumerate}
   \item[] $MF(K_{1})=\{(3,4),(1,2,3),(1,2,4)\}$;
   \item[] $MF(K_{2})=\{(2,4),(3,4),(1,2,3)\}$;
   \item[] $MF(K_{3})=\{(2,5),(3,4),(4,5),(1,2,3),(1,2,4),(1,3,5)\}$.
\end{enumerate} 
Observe that $|K_{1}|=\bigcup_{\sigma\in MF(K_{1})}|\partial\sigma|$
and $|K_{3}|=\bigcup_{\sigma\in MF(K_{3})}|\partial\sigma|$, but
in contrast, 
$\bigcup_{\sigma\in MF(K_{2})}|\partial\sigma|=\mbox{$|K_{2}-(1,4)|$}$. 
As well, $K_{1}$ can be formed by gluing the boundary of $(1,2,3)$ to 
the boundary of $(1,2,4)$ along the common face $(1,2)$, and $K_{3}$ 
can be formed by gluing the boundary of $(1,3,5)$ to~$K_{1}$ along 
the common face $(1,3)$. Thus $K_{1}$ is a shifted complex which is 
also a directed $MF$-complex, while $K_{2}$ is a shifted complex 
which is not a directed $MF$-complex, and $K_{3}$ is a directed  
$MF$-complex which is not shifted. 

As noted in the Introduction, a simple example of an $MF$-complex 
which is not a directed $MF$-complex is the $1$-skeleton 
of a tetrahedron. However, by~\cite{GT1}, in this case the conclusions of 
Theorem~\ref{Mcomplexwedge} and Corollary~\ref{integralM} nevertheless 
hold; in particular, $\zk$ is homotopy equivalent to a wedge of spheres.  
It is useful to also give an example of an $MF$-complex which is not a 
directed $MF$-complex and for which Theorem~\ref{Mcomplexwedge} 
and Corollary~\ref{integralM} fail. Let $K$ be the simplicial complex on~$8$ 
vertices given by 
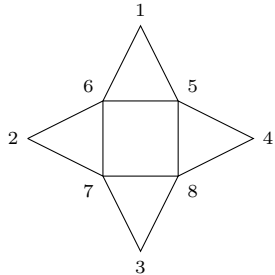
\[\begin{tikzpicture} 
     \draw (0,1.5)--(1,2)--(1,1)--(0,1.5); 
     \draw (1,2)--(1.5,3)--(2,2)--(1,2); 
     \draw (2,2)--(3,1.5)--(2,1)--(2,2); 
     \draw (1,1)--(2,1)--(1.5,0)--(1,1); 
     \node [above] at (1.5,3) {\scriptsize $1$}; 
    \node [left] at (0,1.5) {\scriptsize $2$}; 
    \node [below] at (1.5,0) {\scriptsize $3$}; 
    \node [right] at (3,1.5) {\scriptsize $4$}; 
    \node [above right] at (2,2) {\scriptsize $5$}; 
    \node [above left] at (1,2) {\scriptsize $6$}; 
    \node [below left] at (1,1) {\scriptsize $7$}; 
    \node [below right] at (2,1) {\scriptsize $8$}; 
  \end{tikzpicture}\] 
Notice that $K$ is the union of the boundaries 
of the faces $\{(1,5,6), (2,6,7), (3,7,8), (4,5,8)\}$. This implies that 
$|K|=\bigcup_{\sigma\in MF(K)}|\partial\sigma|$, so $K$ is an 
$MF$-complex. However, $K$ is not a directed $MF$-complex.  
For, thinking of $|K|$ as the union of the boundaries of four triangles, it is 
possible to glue the second triangle to the first and the third to 
the first two along a common vertex, but the fourth triangle is 
glued to the first three along two vertices, which is not a  
common face. In this case the existence of the boundary of the 
square in $K$ leads to nontrivial cup products in the cohomology 
of~$\zk$, implying that \zk\ cannot be a wedge of spheres. 
\medskip 

Now we turn to the homotopy types of $\zk(\underline{X})$ and 
\zk\ when $K$ is a directed $MF$-complex. For clarity, we restate 
Theorem~\ref{Mcomplexwedge} and Corollary~\ref{integralM}. 

\begin{theorem} 
   Let $K$ be a directed $MF$-complex on $n$ vertices. Let
   $\underline{X}=\{X_{1},\ldots,X_{n}\}$ where each $X_{i}$ is
   a path-connected topological space. Then $\zk(\underline{X})$ is
   homotopy equivalent to a wedge of spaces of the form
   $\Sigma^{t}\Omega X_{i_{1}}\wedge\cdots\wedge\Omega X_{i_{k}}$
   for various $1\leq t<n$ and sequences $(i_{1},\ldots,i_{k})$
   where ${1\leq i_{1}<\cdots<i_{k}\leq n}$.
\end{theorem} 

\begin{proof} 
Write $Y\in\mathcal{W}$ if $Y$ is a space that is homotopy equivalent
to a wedge of spaces of the form
$\Sigma^{t}\Omega X_{i_{1}}\wedge\cdots\wedge\Omega X_{i_{k}}$
for various $1\leq t<n$ and sequences $(i_{1},\ldots,i_{k})$
where $1\leq i_{1}<\cdots<i_{k}\leq n$. For any such sequence, let
$FW(i_{1},\ldots,i_{k})$ be the fat wedge of the product
$\prod_{j=1}^{k} X_{i_{j}}$. By definition, the homotopy fibre
of the inclusion
\(\namedright{FW(i_{1},\ldots,i_{k})}{}{\prod_{j=1}^{k} X_{i_{j}}}$
is \zk\ for $K=(\Delta^{k})_{k-1}=\partial\sigma$. On the other hand,
by~\cite{P2}, this homotopy fibre is homotopy equivalent to
$\Sigma^{k-1}\Omega X_{i_{1}}\wedge\cdots\wedge\Omega X_{i_{k}}$.
In particular, in an $MF$-complex
$\vert K\vert=\bigcup_{\sigma\in MF(K)}\vert\partial\sigma\vert$, we have
$\mathcal{Z}_{\partial\sigma}\in\mathcal{W}$ for every $\sigma\in MF(K)$. 

By definition, since $K$ is a directed $MF$-complex, there is a 
sequence of subcomplexes 
$\emptyset\subseteq K_{1}\subseteq\cdots\subseteq K_{l}=K$ 
for some $l$, where $K_{i}=K_{i-1}\cup\partial\sigma_{i}$ for 
$\sigma_{i}\in MF(K)$ and $K_{i-1}\cap\sigma_{i}$ is a face common 
to $K_{i-1}$ and $\sigma_{i}$. We proceed with the proof by induction. 
Since $K_{1}=\partial\sigma_{1}$, the previous paragraph shows that
$\mathcal{Z}_{K_{1}}\in\mathcal{W}$. Now suppose that
$\mathcal{Z}_{K_{i-1}}\in\mathcal{W}$. We have $K_{i}$ constructed
by gluing $\partial\sigma_{i}$ and $K_{i-1}$ along a common face,
$\mathcal{Z}_{\partial\sigma_{i}}\in\mathcal{W}$ by the preceding
paragraph, and $\mathcal{Z}_{K_{i-1}}\in\mathcal{W}$ by assumption.
Under these circumstances, \cite[Theorem 1.3]{GT1} implies that
$Z_{K_{i}}\in\mathcal{W}$. (Actually, \cite[Theorem 1.3]{GT1} is
stated for the special case when $X_{i}=\cpinf$ for $1\leq i\leq n$,
but the proof goes through without change in the general case.)
Hence, by induction, $\zk(\underline{X})=\mathcal{Z}_{K_{l}}$ is in
$\mathcal{W}$. 
\end{proof} 

\begin{corollary} 
   Let $K$ be a directed $MF$-complex on $n$ vertices. Then \zks\ 
   is homotopy equivalent to an infinite wedge of simply-connected 
   spheres, and \zk\ is homotopy equivalent to a finite wedge of
   simply-connected spheres.
\end{corollary}

\begin{proof} 
In Theorem~\ref{Mcomplexwedge}, suppose that $X_{i}=S^{m_{i}+1}$ for
each $1\leq i\leq n$. Then \zks\ is homotopy equivalent to a wedge
of spaces of the form $\Sigma^{t}\Omega
S^{m_{i_{1}}+1}\wedge\cdots\wedge\Omega S^{m_{i_{k}}+1}$ for various
$1\leq t<n$ and sequences $(i_{1},\ldots,i_{k})$. By~\cite{J}, there
is a homotopy equivalence $\Sigma\Omega
S^{m_{i}+1}\simeq\bigvee_{j=1}^{\infty} S^{jm_{i}+1}$. Since
$S^{jm_{i}+1}$ is a suspension, iterating this homotopy equivalence shows that
each wedge summand $\Sigma^{t}\Omega
S^{m_{i_{1}}+1}\wedge\cdots\wedge\Omega S^{m_{i_{k}}+1}$ is homotopy
equivalent to an infinite wedge of simply-connected spheres. Thus \zks\ is
homotopy equivalent to ain infinite wedge of simply-connected spheres.

Next, in Theorem~\ref{Mcomplexwedge}, suppose that $X_{i}=\cpinf$
for each $1\leq i\leq n$. Then \zk\ is homotopy equivalent to a
wedge of spaces of the form
$\Sigma^{t}\Omega\cpinf_{i_{1}}\wedge\cdots\wedge\Omega\cpinf_{i_{k}}$
for various $1\leq t<n$ and sequences $(i_{1},\ldots,i_{k})$. Since
$\Omega\cpinf\simeq S^{1}$, each wedge summand
$\Sigma^{t}\Omega\cpinf_{i_{1}}\wedge\cdots\wedge\Omega\cpinf_{i_{k}}$
is homotopy equivalent to $S^{k+t}$. Thus \zk\ is homotopy
equivalent to a finite wedge of simply-connected spheres. 
\end{proof}

\section{Higher Whitehead products and Fat wedges}
\label{sec:fatwedge}

In this section we define a higher Whitehead product by means
of a fat wedge, and relate the existence of a missing face in $K$
to the existence of a nontrivial higher Whitehead in $DJ_{K}(\underline{X})$.
Let $X_{1},\ldots,X_{n}$ be path-connected spaces and let
$\underline{X}=\{X_{1},\ldots,X_{n}\}$. The fat wedge is the space
\[FW(\underline{X})=\{(x_{1},\ldots,x_{n})\in X_{1}\times\cdots\times X_{n}
      \mid \mbox{at least one}\ x_{i}=\ast\}.\]
Consider the homotopy fibration obtained by including $FW(\underline{X})$
into the product $X_{1}\times\cdots\times X_{n}$. The homotopy type of the
fibre was first identified by Porter~\cite{P2}, who showed that there
is a homotopy fibration
\[\nameddright{\Sigma^{n-1}\Omega X_{1}\wedge\cdots\wedge\Omega X_{n}}
    {}{FW(\underline{X})}{}{X_{1}\times\cdots\times X_{n}}.\] 

There is a reformulation of this in terms of the polyhedral product. 
Let$K=\partial\Delta^{n-1}$, the boundary of the $(n-1)$-simplex. 
Let $CY$ be the cone on $Y$, parameterized as $CY=[0,1]\times Y/\sim$ 
for $(0,y)\sim\ast$ - that is, the tip of the cone is at $0$. Observe that 
there is a map of pairs 
\(\namedright{C\Omega X_{i},\Omega X_{i})}{}{(X_{i},\ast)}\) 
given by sending $(s,\omega)\in C\Omega X_{i}$ to $\omega(s)$. Then, 
essentially by~\cite{P2}, the map 
\(\namedright{\Sigma^{n-1}\Omega X_{1}\wedge\cdots\wedge\Omega X_{n}} 
        {}{FW(\underline{X})}\) 
can be identified with the map 
\(\namedright{(\underline{C\Omega X},\underline{\Omega X})^{K}} 
      {}{(\underline{X},\underline{\ast})^{K}}\). 

If each $X_{i}$ is a suspension, $X_{i}=\Sigma Y_{i}$, then the
suspension map
\(E\colon\namedright{Y}{}{\Omega\Sigma Y}\)
induces a composite
\[\phi_{n}\colon\nameddright{\Sigma^{n-1} Y_{1}\wedge\cdots\wedge Y_{n}}{}
     {\Sigma^{n-1}\Omega\Sigma Y_{1}\wedge\cdots\wedge\Omega\Sigma Y_{n}}
     {}{FW(\underline{\Sigma Y})}.\]
The map $\phi_{n}$ is the attaching map that yields the product. That is,
there is a homotopy cofibration
\[\nameddright{\Sigma^{n-1} Y_{1}\wedge\cdots\wedge Y_{n}}{\phi_{n}}
    {FW(\underline{\Sigma Y})}{}{\Sigma Y_{1}\times\cdots\times\Sigma Y_{n}}.\]

In the case $n=2$, we have $FW(\underline{\Sigma Y})=\Sigma
Y_{1}\vee\Sigma Y_{2}$ and $\phi_{2}$ is the Whitehead product
$[i_{1},i_{2}]$, where $i_{1}$ and $i_{2}$ are the inclusions of
$\Sigma Y_{1}$ and $\Sigma Y_{2}$ respectively into $\Sigma
Y_{1}\vee\Sigma Y_{2}$. This is the universal example for Whitehead
products. Given a space $Z$ and maps \(f\colon\namedright{\Sigma
Y_{1}}{}{Z}\) and \(g\colon\namedright{\Sigma Y_{2}}{}{Z}\), the
Whitehead product $[f,g]$ of $f$ and $g$ is the composite
\(\nameddright{\Sigma Y_{1}\wedge Y_{2}}{\phi_{2}}{\Sigma
Y_{1}\vee\Sigma Y_{2}}
   {f\perp g}{Z}\),
where~$\perp$ denotes the wedge sum. Porter~\cite{P1} used the maps
$\phi_{n}$ for $n>2$ as universal examples to define higher Whitehead products.

\begin{definition}
For $n\geq 2$, let $Y_{1},\ldots,Y_{n}$ and $Z$ be path-connected
spaces, and let
\(f_{i}\colon\namedright{\Sigma Y_{i}}{}{Z}\)
be maps. Suppose that the wedge sum
\(\namedright{\bigvee_{i=1}^{n}\Sigma Y_{i}}{}{Z}\)
of the maps $f_{i}$ extends to a map
\(f\colon\namedright{FW(\underline{\Sigma Y})}{}{Z}\).
The \emph{$n^{th}$-higher Whitehead product} of the maps
$f_{1},\ldots,f_{n}$ is the composite
\[[f_{1},\ldots,f_{n}]\colon\nameddright
    {\Sigma^{n-1} Y_{1}\wedge\cdots\wedge Y_{n}}{\phi_{n}}
    {FW(\underline{\Sigma Y})}{f}{Z}.\]
\end{definition}

If $n=2$, the Whitehead product of two maps $f_{1}$ and $f_{2}$ is
always defined, and the homotopy class of $[f_{1},f_{2}]$ is
uniquely determined by the homotopy classes of $f_{1}$ and $f_{2}$.
If $n>2$, it may not be the case that the higher Whitehead product
of $n$ maps $f_{1},\ldots,f_{n}$ exists, as there may be nontrivial
obstructions to extending the given map
\(\namedright{\bigvee_{i=1}^{n}\Sigma Y_{i}}{}{Z}\) to the fat
wedge $FW(\underline{\Sigma Y})$. Even if such an extension exists, there 
may be many inequivalent choices of an extension, implying that the homotopy
class of $[f_{1},\ldots,f_{n}]$ need not be uniquely determined by
the homotopy classes of $f_{1},\ldots,f_{n}$.

When $n=2$, the adjoint of the Whitehead product $[f_{1},f_{2}]$ is
homotopic to a Samelson product. Its image in homology is given by
commutators. We wish to have analogous information about higher
Whitehead products. The universal example is given by the adjoint
of $\phi_{n}$, which is a map
\(\namedright{\Sigma^{n-2} Y_{1}\wedge\cdots\wedge Y_{n}}{}
    {\Omega FW(\underline{\Sigma Y})}\).
We want to know the Hurewicz image of this map. To do so we need a
good model for $\hlgy{\Omega FW(\underline{\Sigma Y})}$ which sees
this Hurewicz image. Producing such a model in the case when each $Y_{i}$
is a sphere is the subject of Section~\ref{sec:AdamsHilton}.

Before getting to this, we give a general result which identifies
nontrivial higher Whitehead products in \djky\ for any simplicial
complex $K$. In short, there is a nontrivial higher Whitehead
product for each missing face of $K$ which, moreover, lifts to \zky.
In what follows we will consider sub-products $\prod_{j=1}^{k}\Sigma
Y_{i_{j}}$ of $\prod_{i=1}^{n}\Sigma Y$. If
$\sigma=(i_{1},\ldots,i_{k})$, let $FW(\underline{\Sigma Y},\sigma)$
be the fat wedge of the sub-product $\prod_{j=1}^{k}\Sigma
Y_{i_{j}}$. 

\begin{lemma} 
   \label{fgsigmasplit} 
   Let $K$ be a simplicial complex on $n$ vertices. If
   $\sigma=(i_{1},\ldots,i_{k})\in MF(K)$ then there exists maps 
   \(f_{\sigma}\colon\namedright{FW(\underline{\Sigma Y},\sigma)}{}{\djky}\) 
   and 
   \(g_{\sigma}\colon\namedright{\Sigma^{k-1} Y_{i_{1}}\wedge\cdots\wedge Y_{i_{k}}} 
           {}{\zky}\) 
    which both have left homotopy inverses, and which fit into a 
    homotopy commutative diagram 
   \[\diagram 
          \Sigma^{k-1} Y_{i_{1}}\wedge\cdots\wedge Y_{i_{k}} 
                 \rto^-{\phi_{k}}\dto^{g_{\sigma}} 
              & FW(\underline{\Sigma Y})\dto^{f_{\sigma}} \\ 
          \zky\rto & \djky. 
     \enddiagram\] 
\end{lemma} 

\begin{proof} 
Since $\sigma\in MF(K)$, it is the full subcomplex of $K$ on the vertex 
set $\{i_{1},\ldots,i_{k}\}$. Therefore, if $(\underline{X},\underline{A})$ 
is any $n$ pairs of $CW$-complexes $(X_{i},A_{i})$ then the polyhedral product 
$(\underline{X},\underline{A})^{\sigma}$ is a natural retract of 
$(\underline{X},\underline{A})^{K}$. Applying this to the map 
\(\namedright{(\underline{C\Omega\Sigma Y},\underline{\Omega\Sigma Y})} 
     {}{(\underline{\Sigma Y},\underline{\ast})}\) 
obtained from the maps of pairs 
\(\namedright{(C\Omega\Sigma Y_{i},\Omega\Sigma Y_{i})}{} 
      {(\Sigma Y_{i},\ast)}\), 
we obtain a homotopy commutative diagram 
\[\diagram 
       \Sigma^{k-1}\Omega\Sigma Y_{i_{1}}\wedge\cdots\wedge\Omega\Sigma Y_{i_{k}} 
              \rto\dto^{g^{\prime}_{\sigma}} 
           & FW(\underline{\Sigma Y})\dto^{f_{\sigma}} \\ 
       \zky\rto & \djky 
  \enddiagram\]
where $f_{\sigma}$ and $g^{\prime}_{\sigma}$ have left homotopy 
inverses. Precompose the diagram with the map 
\[\epsilon\colon\namedright{\Sigma^{k-1} Y_{i_{1}}\wedge\cdots\wedge Y_{i_{k}}} 
     {}{\Sigma^{k-1}\Omega\Sigma Y_{i_{1}}\wedge\cdots\wedge\Omega\Sigma Y_{i_{k}}}\] 
induced by the suspension maps 
\(\namedright{Y_{i}}{E}{\Omega\Sigma Y_{i}}\). 
By definition, $\phi_{k}$ is $\epsilon$ composed with the upper horizontal 
map in the preceding diagram, so if let 
$g_{\sigma}=g^{\prime}_{\sigma}\circ\epsilon$ 
then we obtain the homotopy commutative diagram asserted by the lemma. 

It remains to show that $g_{\sigma}$ has a left homotopy inverse. 
But as $\Sigma E$ has a left homotopy inverse, so does $\epsilon$, 
and therefore as $g^{\prime}_{\sigma}$ has a left homotopy inverse, 
so does $g_{\sigma}$. 
\end{proof}

\section{Adams-Hilton models}
\label{sec:AdamsHilton}

Let $X$ be a simply-connected $CW$-complex of finite type and $R$ a
commutative ring. The Adams-Hilton
model~\cite{AH} is a means of calculating $\hlgy{\Omega X;R}$. One advantage
it has over other models for $\hlgy{\Omega X;R}$ is its relative simplicity,
which allows for concrete calculations in certain cases. Its presentation is
stated in Theorem~\ref{AH}.

Let $V$ be a graded $R$-module, and
let $T(V)$ be the free tensor algebra on $V$. For a space $X$, let
$CU_{\ast}(X)$ be the cubical singular chain complex on $X$ with
coefficients in $R$. Note that $CU_{\ast}(X)$ is naturally chain equivalent
to the simplicial singular chain complex on~$X$. If $X$
is a homotopy associative $H$-space then the multiplication on $X$
induces a multiplication on $CU_{\ast}(X)$, giving it the structure
of a differential graded algebra. A map
\(\namedright{A}{}{B}\)
of differential graded algebras is a \emph{quasi-isomorphism} if
it induces an isomorphism in homology.

\begin{theorem}
   \label{AH}
   Let $R$ be a commutative ring and let $X$ be a simply-connected
   $CW$-complex of finite type. The Adams-Hilton model for $X$
   is a differential graded $R$-algebra $AH(X)$ satisfying:
   \begin{letterlist}
      \item if $X=pt\ \cup(\bigcup_{\alpha\in S} e_{\alpha})$ is a
            $CW$-decomposition of $X$ then
            $AH(X)=T(V; d_{V})$ where $V=\{b_{\alpha}\}_{\alpha\in S}$
            and $\vert b_{\alpha}\vert=\vert e_{\alpha}\vert -1$;
      \item the differential $d_{V}$ depends on the attaching maps
            of the $CW$-complex $X$;
      \item there is a map of differential graded algebras
            \(\theta_{X}\colon\namedright{AH(X)}{}{CU_{\ast}(\Omega X)}\)
            which induces an isomorphism
            $\hlgy{AH(X)}\cong\hlgy{\Omega X;R}$.
   \end{letterlist}
   $\qqed$
\end{theorem}

Notice that the generators of $AH(X)$ are in one-to-one
correspondence with the cells of $X$, shifted down by one dimension.
However, the differential $d_{V}$ and the quasi-isomorphism
$\theta_{X}$ are not uniquely determined by the $CW$-structure of
$X$. There may be many inequivalent choices of both $d_{V}$
and~$\theta_{X}$ which result in an isomorphism
$\hlgy{AH(X)}\cong\hlgy{\Omega X;R}$. In that sense, there may be
many Adams-Hilton models for $\hlgy{\Omega X;R}$. One would hope to
choose a model which is particularly advantageous. This is what we
aim to do for $X=\djks$ or \djk\ and $R=\mathbb{Q}$, by choosing a
model which keeps track of the Hurewicz images of adjointed higher
Whitehead products.

We start with some general constructions in the case of \djkstwo\
and \djk, producing Adams-Hilton models for both which are compatible
with the inclusion
\(\namedright{S^{2}}{\imath}{\cpinf}\)
of the bottom cell. The model will then be generalized to \djks,
but without the need for an accompanying map.

By definition, for $\sigma=(i_{1},\ldots,i_{k})$, let
$S^{\sigma}=\prod_{j=1}^{k} S^{2}_{i_{j}}$, with the lower index
recording coordinate position, and let $\djkstwo=\bigcup_{\sigma\in
K} S^{\sigma}$. Similarly, regarding $\cpinf$ as $BT$ where
$T=S^{1}$, by definition $\djk=\bigcup_{\sigma\in K} BT^{\sigma}$,
where $BT^{\sigma}=BT_{i_{1}}\times\cdots\times BT_{i_{k}}$, again
with the lower index recording coordinate position. Let
\(\imath^{\sigma}\colon\namedright{S^{\sigma}}{}{BT^{\sigma}}\) be
the product map $\prod_{j=1}^{k}\imath$. Then the map
\(\namedright{\djks}{\djk(\imath)}{\djk}\) is, by definition,
$\bigcup_{\sigma\in K}\imath^{\sigma}$.

Many useful properties of Adams-Hilton models were proved
in~\cite{AH}; a nice summary can be found in~\cite[8.1]{An}. First,
an Adams-Hilton model of a $CW$-subspace can be extended to one for
the whole space. Start with the inclusion of the bottom cell
\(\namedright{S^{2}}{\imath}{\cpinf}\). Then an Adams-Hilton model
for $S^{2}$ can be extended to one for $\cpinf$. Second, an
Adams-Hilton model for a product $AH(X\times Y)$ can be chosen so
that it is quasi-isomorphic to $A(X)\otimes A(Y)$, and this respects
the quasi-isomorphisms $\theta_{X\times Y}$ and
$\theta_{X}\otimes\theta_{Y}$ to the respective cubical singular
chain complexes. In our case, this lets us take the given model
$AH(S^{2})$ for $S^{2}$ and its extension $AH(\cpinf)$ for \cpinf\
and produce a model for $S^{\sigma}$ mapping to $BT^{\sigma}$ which,
up to quasi-isomorphisms, is $AH(S^{2})^{\otimes\sigma}$ mapping
factor-wise to $AH(\cpinf)^{\otimes\sigma}$. Third, Adams-Hilton
models preserve colimits, given coherency conditions. That is, if
$\{X_{\alpha}\}$ is a family of $CW$-subcomplexes of $X$ and
$X=\bigcup_{\alpha} X_{\alpha}$, and there are models
$AH(X_{\alpha})$ satisfying the coherency conditions
$d_{V_{\alpha}}\vert_{AH(X_{\alpha}\cap X_{\beta})}=
      d_{V_{\beta}}\vert_{AH(X_{\alpha}\cap X_{\beta})}$
and
$\theta_{X_{\alpha}}\vert_{AH(X_{\alpha}\cap X_{\beta})}=
      \theta_{X_{\beta}}\vert_{AH(X_{\alpha}\cap X_{\beta})}$
for all pairs $(\alpha,\beta)$, then $\colim_{\alpha} AH(X_{\alpha})$ is
an Adams-Hilton model for $X$. In our case, we have
\(\llnamedright{\djks}{\djk(\imath)}{\djk}\)
equalling, by definition,
\(\lllnamedright{\bigcup_{\sigma\in K} S^{\sigma}}
     {\bigcup_{\sigma\in K}\imath^{\sigma}}{\bigcup_{\sigma\in K} BT^{\sigma}}\).
Notice that the intersection $S^{\sigma_{1}}\cap S^{\sigma_{2}}$ is again
a sub-product, namely $S^{\sigma_{1}\cap\sigma_{2}}$. Similarly,
$BT^{\sigma_{1}}\cap BT^{\sigma_{2}}=BT^{\sigma_{1}\cap\sigma_{2}}$.
Thus the compatibility of Adams-Hilton models with products implies
that the coherency conditions will be satisfied for
$\bigcup_{\sigma\in K} S^{\sigma}$ and $\bigcup_{\sigma\in K} BT^{\sigma}$,
and for the map $\bigcup_{\sigma\in K} i^{\sigma}$. Hence there is a 
commutative diagram 
\begin{equation} 
  \label{colimAHK} 
  \diagram
      AH(\djks)\rto^-{=}\dto^{AH(\djk(\imath))}
         & \colim_{\sigma\in K} AH(S^{\sigma})
             \dto^{\colim_{\sigma\in K} AH(\imath^{\sigma})}  \\
      AH(\djk)\rto^-{=} & \colim_{\sigma\in K} AH(BT^{\sigma}).
  \enddiagram 
\end{equation} 

At this point one would like to say that homology commutes with colimits 
in order to describe $\qhlgy{\Omega\djkstwo}$, say, as 
$\colim_{\sigma\in K}\qhlgy{\Omega S^{\sigma}}$. But there is a problem: 
$AH(\djks)$ is a noncommutative differential graded algebra and in general, 
colimits of such objects do not commute with homology. The point is that 
the colimit does not see higher bracket terms that may arise from the 
interaction of the differential and noncommutativity. Instead, one needs 
to take an appropriate homotopy colimit. However, in the case of directed 
$MF$-complexes, we are able to avoid this problem. 

To see this, consider an analogue of~(\ref{colimAHK}) with respect 
to directed $MF$-complexes and fat wedges. To distinguish fat wedges,
given $\sigma=(i_{1},\ldots,i_{k})$, let $FW(S^{2},\sigma)$ be the
fat wedge of $\prod_{j=1}^{k} S^{2}_{i_{j}}$, where the lower index
refers to coordinate position. Let $FW(\sigma)$ be the fat wedge of
$\prod_{j=1}^{k}\cpinf_{i_{j}}$. Let
\(\imath^{\sigma}\colon\namedright{FW(S^{2},\sigma)}{}{FW(\sigma)}\)
be the map of fat wedges induced by $\imath$. Note that
$FW(S^{2},\sigma)=\bigcup_{\tau\in(\Delta^{k})_{k-1}} S^{\tau}$ and
$FW(\sigma)=\bigcup_{\tau\in(\Delta^{k})_{k-1}} BT^{\tau}$. Suppose
$K$ is a directed $MF$-complex on $n$ vertices. The fact that $K$ 
is an $MF$-complex (missing face complex) implies that 
\[\djks=\bigcup_{\sigma\in MF(K)} FW(S^{2},\sigma)=
     \bigcup_{\sigma\in MF(K)}\bigcup_{\tau\in(\Delta^{k})_{k-1}} S^{\tau}.\] 
Since $K$ is a directed $MF$-complex, then in addition there 
is a sequence of subcomplexes 
$\emptyset\subseteq K_{1}\subseteq\cdots\subseteq K_{l}=K$ 
where $K_{i}=K_{i-1}\cup\partial\sigma_{i}$ and 
$K_{i-1}\cap\partial\sigma_{i}$ is a face common to $K_{i-1}$ 
and $\partial\sigma_{i}$. Suppose this common face is 
$(t_{1},\ldots,t_{l})$. Then topologically there is a (strict) pushout 
\[\diagram 
         \prod_{j=1}^{l} S^{2}_{t_{j}}\rto\dto & DJ_{K_{i-1}}(S^{2})\dto \\ 
         FW(S^{2},\sigma_{i})\rto & DJ_{K_{i}}(S^{2}). 
  \enddiagram\] 
That is, $DJ_{K_{i-1}}(S^{2})$ and $FW(S^{2},\sigma_{i})$ have been 
glued together over the sub-product $\prod_{j=1}^{t} S^{2}_{t_{j}}$ 
of $\prod_{i=1}^{n} S^{2}$. This pushout satisfies the coherency 
conditions for the Adams-Hilton model. The same is true for 
\djk\ and the map $\djk(\iota)$. Thus, in the case of directed  
$MF$-complexes, (\ref{colimAHK}) can be reformulated as a commutative 
diagram 
\begin{equation} 
  \label{colimAHFW} 
  \diagram
      AH(\djks)\rto^-{=}\dto^{AH(\djk(\imath))}
         & \colim_{\sigma\in MF(K)} AH(FW(S^{2},\sigma))
             \dto^{\colim_{\sigma\in MF(K)} AH(\imath^{\sigma})}  \\
      AH(\djk)\rto^-{=} & \colim_{\sigma\in MF(K)} AH(FW(\sigma)). 
  \enddiagram 
\end{equation} 

Since $K$ is a directed $MF$-complex, $K_{i}$ is obtained by gluing $K_{i-1}$ 
and $\partial\sigma_{i}$ along a common face. Thus the Adams-Hilton model 
$AH(DJ_{K_{i}})$ is a free extension of of the differential graded algebras 
$AH(DJ_{K_{i-1}})$ and $AH(FW(\sigma_{i}))$ and consequently, the 
diagram~(\ref{colimAHFW}) is Reedy cofibrant (see~\mbox{\cite[Sections 3,4]{PR}} for a 
discussion). Therefore, by~\cite[Proposition 4.8]{PR} the colimits in~(\ref{colimAHFW}) 
are naturally weakly equivalent to homotopy colimits. As homotopy colimits 
of differential graded algebras commute with homology, we immediately obtain 
the following. 

\begin{proposition}
   \label{MdjkAH}
   Let $K$ be a directed $MF$-complex. There is a commutative diagram of algebras
   \[\diagram
         \qhlgy{\Omega\djkstwo}\rto^-{\cong}\dto^{(\Omega\djk(\imath))_{\ast}}
             & \colim_{\sigma\in MF(K)}\qhlgy{\Omega FW(S^{2},\sigma)}
                  \dto^{\colim_{\sigma\in MF(K)}(\Omega\imath^{\sigma})_{\ast}} \\
         \qhlgy{\Omega\djk}\rto^-{\cong}
             & \colim_{\sigma\in MF(K)}\qhlgy{\Omega FW(\sigma)}
     \enddiagram\]
   $\qqed$
\end{proposition}

Proposition~\ref{MdjkAH} reduces the problem of calculating
$\qhlgy{\Omega\djkstwo}$ and $\qhlgy{\Omega\djk}$ to that of calculating
the rational homology of looped fat wedges -- with the proviso that
the underlying model for the fat wedges must be compatible with
the inclusion of sub-products. In our case we want more, that the
homology of the looped fat wedges also keeps track of the Hurewicz
images of adjointed higher Whitehead products. We will discuss this
further in the next section.

Observe that the arguments above are equally valid for $\djk(\underline{X})$,
where $\underline{X}=\{X_{1},\ldots,X_{n}\}$. The cases we particularly
care about are $\underline{S}=\{S^{m_{1}+1},\ldots,S^{m_{n}+1}\}$,
with the special case of \djkstwo, and the case of \djk. The
focus is on these cases as they give models which can be explicitly
calculated. We do so for \djks\ in Section~\ref{sec:djksproperties}
and \djk\ in Section~\ref{sec:djkproperties}. For future reference,
we state the case for $\underline{S}$.

\begin{proposition}
   \label{MdjksAH}
   Let $K$ be a directed $MF$-complex. An Adams-Hilton model for \djks\ is
   \[AH(\djks)=\colim_{\sigma\in MF(K)} AH(FW(\sigma))\] 
   and there is an isomorphism 
   \[\hspace{4cm}\qhlgy{\Omega\djks}\cong 
        \colim_{\sigma in MF(K)}\qhlgy{\Omega FW(\sigma)}.\hspace{3cm}\qqed\] 
\end{proposition}

\section{An Adams-Hilton model for $FW(\underline{S})$}
\label{sec:AdamsHiltonFW}

We are aiming for an Adams-Hilton model for $FW(\underline{S})$
over $\mathbb{Q}$ which is compatible with the inclusion of
sub-products, and which keeps track of the Hurewicz images
of adjointed higher Whitehead products. We will obtain one
by using Allday's construction of a minimal Quillen model
for $\pi_{\ast}(\Omega FW(\underline{S}))\otimes\mathbb{Q}$
and then using this to produce an Adams-Hilton model for
$\qhlgy{\Omega FW(\underline{S})}$.

We begin with some general statements which hold for any path-connected
space $X$. Assume from now on that the ground ring $R$ is $\mathbb{Q}$.
Observe that $\pi_{\ast}(X)$ can be given the structure of a graded
Lie algebra  by using the Whitehead product to define the bracket.
Equivalently, by adjointing, $\pi_{\ast}(\Omega X)$ may be given the
structure of a graded Lie algebra by using the Samelson product.
Quillen~\cite{Q} associated to $X$ a free differential graded Lie
algebra $\lambda(X)$ over~$\mathbb{Q}$ with the property that there
is an isomorphism
\(\namedright{\hlgy{\lambda(X)}}{}
    {\pi_{\ast}(\Omega X)\otimes\mathbb{Q}}\).
The free property of~$\lambda(X)$ lets us write it as
$L\langle V;d_{V}\rangle$ for some graded $\mathbb{Q}$-module $V$
and differential $d_{V}$ on $V$. A Quillen model $MQ(X)$ is
\emph{minimal} if the differential has the property that
$d(L\langle V\rangle)\subseteq [L\langle V\rangle,L\langle V\rangle]$.

Allday~\cite{Al} gave an explicit construction of a minimal Quillen
model for $\pi_{\ast}(\Omega FW(\underline{S}))\otimes\mathbb{Q}$.
This is stated in Theorem~\ref{AlldayQuillen} once some notation has
been introduced. The cells of $FW(\underline{S})$ are in one-to-one
correspondence with sequences $(i_{1},\ldots,i_{k})$ where $1\leq
i_{1}<\cdots<i_{k}\leq n$ and $k<n$. The sequence
$(i_{1},\ldots,i_{k})$ corresponds to the top cell of the coordinate
subspace $S^{m_{i_{1}}+1}\times\cdots\times S^{m_{i_{k}}+1}$ within
$FW(\underline{S})$. This cell has dimension $\Sigma_{s=1}^{k}
(m_{i_{s}}+1)$. Note that the condition $k<n$ excludes only one
sequence, $(1,2,\ldots,n)$, corresponding to the top cell of the
product $S^{m_{1}+1}\times\cdots\times S^{m_{n}+1}$. Allday's
minimal Quillen model for $\pi_{\ast}(\Omega
FW(\underline{S}))\otimes\mathbb{Q}$ is of the form
\[MQ(FW(\underline{S}))=L\langle V;d_{V}\rangle\]
where $V$ has one generator $b_{I}$ for each sequence
$I=(i_{1},\ldots,i_{k})$ with $1\leq i_{1}<\cdots<i_{k}\leq n$
and $k<n$, and the degree of $b_{I}$ is
$(\Sigma_{s=1}^{k} (m_{i_{s}}+1))-1$. 
Similarly, his minimal Quillen model for
$\pi_{\ast}(\Omega\prod_{i=1}^{n} S^{m_{i}+1})\otimes\mathbb{Q}$
is of the form
\[MQ(\prod_{i=1}^{n} S^{m_{i}+1})=L\langle W;d_{W}\rangle\]
where $W$ has one generator $b_{I}$ for each sequence
$I=(i_{1},\ldots,i_{k})$ with $1\leq i_{1}<\cdots<i_{k}\leq n$,
and the degree of $b_{I}$ is $(\Sigma_{s=1}^{k} (m_{i_{s}}+1))-1$.

To describe the differentials $d_{V}$ and $d_{W}$, fix a sequence
$I=(i_{1},\ldots,i_{k})$ where $1\leq i_{1}<\cdots<i_{k}\leq n$ and
$k\geq 2$. If $k<n$ this corresponds to a generator $b_{I}$ of $V$,
and if $k\leq n$ this corresponds to a generator $b_{I}$ of $W$. In
either case, the degree of $b_{I}$ is $\vert
b_{I}\vert=(\Sigma_{s=1}^{k} m_{i_{s}}+1)-1$. Let $\mathcal{S}_{I}$
be the collection of all shuffles $(J,J^{\prime})$ of
$\{i_1,\ldots,i_k\}$ with the property that $j_{1}=1$ (known as a
type II shuffle relative to 1). If $(J,J^{\prime})$ is an
$(r,s)$-shuffle of $\{1,\dots,k\}$, let
$\epsilon(J,J^{\prime})\in\{0,1\}$ be the number determined by the
equation $z_{i_{1}}\cdots z_{i_{k}}=(-1)^{\epsilon(J,J^{\prime})}
    z_{j_{1}}\cdots z_{j_{r}}z_{j^{\prime}_{1}}\cdots z_{j^{\prime}_{s}}$
in the graded rational symmetric algebra generated by
$z_{i_{1}},\ldots,z_{i_{k}}$ with $\vert z_{i_{t}}\vert=m_{i_{t}}+1$
for $1\leq t\leq k$. Let
\[a_{I}=-\displaystyle\Sigma_{(J,J^{\prime})\in\mathcal{S_{I}}}
         (-1)^{\vert b_{I}\vert+\epsilon(J,J^{\prime})}[b_{J},b_{J^{\prime}}].\]
As special cases, let $b=b_{(1,\ldots,n)}$ and $a=a_{(1,\ldots,n)}$.

\begin{theorem}
   \label{AlldayQuillen}
   With $V$ and $W$ as defined above, minimal Quillen models
   $L\langle V;d_{V}\rangle$ and $L\langle W;d_{W}\rangle$ for
   $\pi_{\ast}(\Omega FW(\underline{S}))\otimes\mathbb{Q}$
   and $\pi_{\ast}(\prod_{i=1}^{n}\Omega S^{m_{i}+1})\otimes\mathbb{Q}$
   can be chosen to satisfy the following properties:
   \begin{letterlist}
      \item $W=V\oplus\{b\}$;
      \item $d_{V}(b_{I})=0$ if $I=(i)$ for $1\leq i\leq n$;
      \item $d_{V}(b_{I})=a_{I}$ for $I=(i_{1},\ldots,i_{k})$ with $2\leq k<n$;
      \item $d_{W}$ restricted to $V$ is $d_{V}$;
      \item $d_{W}(b)=a$;
      \item the adjoint of the higher order Whitehead product
            \(\namedright{S^{\vert b\vert -1}}{\phi_{n}}{FW(\underline{S})}\)
            which attaches the top cell to the product
            $\prod_{i=1}^{n} S^{m_{i}+1}$ is homotopic to $a$.
   \end{letterlist}
   $\qqed$
\end{theorem}

There is an explicit map
\(\alpha\colon\namedright{L\langle V\rangle}{}
    {\pi_{\ast}(\Omega FW(\underline{S}))}\).
Let $b_{I}\in V$ for $I=(i_{1},\ldots,i_{k})$. This corresponds to the
top cell of the coordinate subspace
$S^{m_{i_{1}}+1}\times\cdots\times S^{m_{i_{k}}+1}$ in
$FW(\underline{S})$. Let $FW(i_{1},\ldots,i_{k})$ be the fat wedge
in $S^{m_{i_{1}}+1}\times\cdots\times S^{m_{i_{k}}+1}$. Let $\alpha(b_{I})$
be the adjoint of the composite
\(\namedddright{S^{\vert b_{I}\vert-1}}{\phi_{k}}{FW(i_{1},\ldots,i_{k})}
    {}{\prod_{j=1}^{k} S^{m_{i_{j}}+1}}{}{FW(\underline{S})}\),
where the latter two maps are the inclusions. Now extend $\alpha$ to
$L\langle V\rangle$ by using the fact that $\pi_{\ast}(\Omega
FW(\underline{S}))\otimes\mathbb{Q}$ is a Lie algebra under the
Samelson product. Allday's statement that $L\langle V;d_{V}\rangle$
is a minimal Quillen model for $\pi_{\ast}(\Omega
FW(\underline{S})\otimes\mathbb{Q}$ says two things: first, that
$\alpha$ can be upgraded from a map of Lie algebras to a map of
differential graded Lie algebras, where the differential on
$\pi_{\ast}(\Omega FW(\underline{S}))\otimes\mathbb{Q}$ is zero, and
second, that this upgraded map induces an isomorphism in homology. A
similar construction can be made with respect to $\prod_{i=1}^{n}
S^{m_{i}+1}$.

We now pass from a minimal Quillen model to an Adams-Hilton model.
In general, observe that the Hurewicz homomorphism
\(\namedright{\pi_{\ast}(\Omega X)\otimes\mathbb{Q}}{}{\qhlgy{\Omega X}}\)
factors as the composite
\(\nameddright{\pi_{\ast}(\Omega X)\otimes\mathbb{Q}}{c}{CU_{\ast}(\Omega X)}
    {h}{\qhlgy{\Omega X}}\),
where $CU_{\ast}(\Omega X)$ is the cubical singular chain complex
with coefficients in~$\mathbb{Q}$, $c$~is the canonical map to the
cubical singular chains, and $h$ is the quotient map to the homology of the
chain complex. Note that $h$ is an algebra map. Let $MQ(X)$ be a minimal Quillen 
model for $\pi_{\ast}(\Omega X)\otimes\mathbb{Q}$, and suppose there is
an associated map of differential graded Lie algebras
\(\alpha\colon\namedright{MQ(X)=L\langle V_{X};d_{V_{X}}\rangle}{}
    {\pi_{\ast}(\Omega X)\otimes\mathbb{Q}}\)
which induces an isomorphism in homology. Since
$CU_{\ast}(\Omega X)$ is a differential graded algebra, the
composite $c\circ\alpha$ extends to a map
\(\theta_{X}\colon\namedright{UL\langle V_{X};d_{V_{X}}\rangle}{}
     {CU_{\ast}(\Omega X)}\)
of differential graded algebras. Thus there is a commutative diagram
\begin{equation}
  \label{quismconstruction}
  \diagram
    L\langle V_{X};d_{V_{X}}\rangle\rto^-{i}\dto^{\alpha}
       & UL\langle V_{X};d_{V_{X}}\rangle\dto^{\theta_{X}} & \\
   \pi_{\ast}(\Omega X)\otimes\mathbb{Q}\rto^-{c}
       & CU_{\ast}(\Omega X)\rto^-{h}
       & \qhlgy{\Omega X}
  \enddiagram
\end{equation}
where $i$ is the inclusion. By Milnor-Moore~\cite{MM}, regarding
$\pi_{\ast}(\Omega X)\otimes\mathbb{Q}$ as a Lie algebra, we have
$\qhlgy{\Omega X}\cong U(\pi_{\ast}(\Omega X)\otimes\mathbb{Q})$,
with the isomorphism induced by the Hurewicz homomorphism. On the
other hand, $h$ is a map of differential graded algebras once
$\qhlgy{\Omega X}$ has been given the zero differential. Thus
$h\circ\theta_{X}$ is a map of differential graded algebras, and
therefore it is determined by its restriction to the generating set
$V$. The commutativity of the diagram implies that
$h\circ\theta_{X}\vert_{V}=h\circ c\circ\alpha\vert_{V}=:q$. Thus
$h\circ\theta_{X}=U(q)$, implying that $h\circ\theta_{X}$ induces an
isomorphism in rational homology. Hence $UL\langle
V_{X};d_{V_{X}}\rangle$ together with the quasi-isomorphism
$\theta_{X}$ is an Adams-Hilton model for $X$.

In our case, we obtain Adams-Hilton models $(UL\langle
V;d_{V}\rangle,\theta_{FW})$ and $(UL\langle
W;d_{W}\rangle,\theta_{\prod})$ for $FW(\underline{S})$ and
$\prod_{i=1}^{n} S^{m_{i}+1}$, respectively.
Theorem~\ref{AlldayQuillen} therefore implies the following.

\begin{theorem}
   \label{Allday}
   The Adams-Hilton models
   $AH(FW(\underline{S}))=(UL\langle V;d_{V}\rangle,\theta_{FW})$
   and $AH(\prod_{i=1}^{n} S^{m_{i}+1})=(UL\langle W;d_{W}\rangle,\theta_{\prod})$
   have the following properties:
   \begin{letterlist}
      \item $W=V\oplus\{b\}$;
      \item $d_{V}(b_{I})=0$ if $I=(i)$ for $1\leq i\leq n$;
      \item $d_{V}(b_{I})=a_{I}$ for $I=(i_{1},\ldots,i_{k})$ with $2\leq k<n$;
      \item $d_{W}$ restricted to $AH(FW(\underline{S}))$ is $d_{V}$;
      \item $d_{W}(b)=a$;
      \item the adjoint of the higher order Whitehead product
            \(\namedright{S^{\vert b\vert -1}}{\phi_{n}}{FW(\underline{S})}\)
            which attaches the top cell to the product
            $\prod_{i=1}^{n} S^{m_{i}+1}$ has ``Hurewicz" image $a$.
   \end{letterlist}
   $\qqed$
\end{theorem}

\begin{remark}
\label{Alldayremark}
The inductive definition of the differential $d_{V}$ in
the minimal Quillen model $L\langle V;d_{V}\rangle$ for
$FW(\underline{S})$ in Theorem~\ref{Allday} implies
that the differential is compatible with the inclusion
of sub-products. The same is therefore true in
$UL\langle V;d_{V}\rangle$. Moreover, the differential $d_{V}$
is what turns the map $\alpha$ into a map of differential
graded Lie algebras, and so both $\alpha$ and its extension to
the quasi-isomorphism $\theta_{FW}$ are compatible with
the inclusion of sub-products.
\end{remark}

As well as the inductive nature of the Adams-Hilton model,
Theorem~\ref{Allday} also explicitly identifies the ``Hurewicz" image
of the adjoint of the higher order Whitehead product~$\phi_{n}$. We put
Hurewicz in quotes as this image is an element in an Adams-Hilton
model, whereas the honest Hurewicz image is obtained after taking homology.
That is, $a$ is a cycle in $AH(FW(\underline{S}))$, which could also be
a boundary. This turns out not to be the case. Observe that there is
a sequence of isomorphisms
$\hlgy{AH(FW(\underline{S}))}\cong\hlgy{UL\langle V;d_{V}\rangle}\cong
    U(\hlgy{L\langle V;d_{V}\rangle})$,
since homology commutes with the universal enveloping algebra
functor. To calculate $\hlgy{L\langle V;d_{V}\rangle}$ we proceed
exactly as in~\cite{B}, where Bubenik used separated Lie models to
elegantly obtain the answer. This is stated in Theorem~\ref{Bubenik}
in terms of the universal enveloping algebra rather than the Lie
algebra as we are ultimately after $\qhlgy{\Omega
FW(\underline{S})}$. To state the result we need to introduce more
notation. For $1\leq i\leq n$, let $b_{i}$ be the generator in $V$
(or $W$) corresponding to the sequence $I=(i)$. That is, $b_{i}$
corresponds to the cell $S^{m_{i}+1}$ in
$S^{m_{1}+1}\times\cdots\times S^{m_{n}+1}$. Let
$N=(\Sigma_{i=1}^{k} m_{i}+1)-2$. Recall from the Introduction that, 
$L_{ds}\langle b_{1},\ldots,b_{n}\rangle=\oplus_{i=1}^{n} L\langle b_{i}\rangle$.  
Observe that
$\hlgy{AH(\prod_{j=1}^{k} S^{m_{j}+1})}\cong
   UL_{ds}\langle b_{1},\ldots,b_{n}\rangle$.

\begin{theorem}
   \label{Bubenik}
   For $n\geq 3$, there are algebra isomorphisms
   \[\qhlgy{\Omega FW(\underline{S})}\cong\hlgy{AH(FW(\underline{S}))}\cong
        U(L_{ds}\langle b_{1},\ldots,b_{n}\rangle\textstyle\coprod
        L\langle u\rangle)\]
   where $u$, of degree $N$, is the Hurewicz image of the adjoint
   of the higher Whitehead product~$\phi_{n}$. Further, the looped inclusion
   \(\namedright{\Omega FW(\underline{S})}{}
       {\prod_{i=1}^{n} \Omega S^{m_{i}+1}}\)
   is modelled by the map
   \[\namedright{U(L_{ds}\langle b_{1},\ldots,b_{n}\rangle\textstyle\coprod
        L\langle u\rangle)}{U(\pi)}{UL_{ds}\langle b_{1},\ldots,b_{n}\rangle}\]
   where $\pi$ is the projection.~$\qqed$
\end{theorem}

Note that the calculation of $\qhlgy{\Omega FW(\underline{S})}$ is not new,
it was originally done by Lemaire~\cite{L}. What is new and important to keep in
mind about Theorem~\ref{Bubenik} is that the calculation also keeps track
of the Hurewicz image of the adjointed higher Whitehead product $\phi_{n}$.

\begin{remark}
   \label{Bubremark}
   When $n=2$, we have $FW(\underline{S})=S^{m_{1}+1}\vee S^{m_{2}+1}$ and
   then it is well known that
   $\qhlgy{\Omega FW(\underline{S})}\cong\qhlgy{\Omega(S^{m_{1}+1}\vee S^{m_{2}+1})}
       \cong UL\langle b_{1},b_{2}\rangle$.
   In this case $\phi_{2}$ is the ordinary Whitehead product and its adjoint
   has Hurewicz image $u=[b_{1},b_{2}]$. In this case we can regard
   $L\langle b_{1},b_{2}\rangle$ as
   $L_{ds}\langle b_{1},b_{2}\rangle\textstyle\coprod L\langle u\rangle$,
   modulo Jacobi identities on brackets of the form
   $[u,-]=[[b_{1},b_{2}],-]$.
\end{remark}

\section{Properties of $\Omega\djks$ and $\Omega\zks$ for
        directed $MF$-complexes}
\label{sec:djksproperties}

In this section we explicitly calculate $\qhlgy{\Omega\djks}$
when $K$ is a directed $MF$-complex, proving Theorem~\ref{ULcolim}.
This is then used in tandem with the
loops on the homotopy fibration
\(\nameddright{\zks}{f}{\djks}{g}{\prod_{i=1}^{n} S^{m_{i}+1}}\)
to calculate $\qhlgy{\Omega\zks}$. We then give a homotopy
decomposition of \zks\ as a wedge of spheres and describe the map
\(\namedright{\zks}{}{\djks}\)
in terms of higher Whitehead products and iterated Whitehead products,
proving Theorem~\ref{djksmaps}.

\begin{remark} 
\label{simplifyremark} 
To simplify the presentation, for the remainder of
Sections~\ref{sec:djksproperties} and~\ref{sec:djkproperties} we will 
assume that the given directed $MF$-complex $K$ has the property that
$\vert\sigma\vert>1$ for every $\sigma\in MF(K)$. This is to appeal
directly to Theorem~\ref{Bubenik}. If $\vert\sigma\vert=1$ for
some $\sigma=(i_{1},i_{2})\in MF(K)$, then the calculations can be
modified by regarding $L\langle b_{i_{1}},b_{i_{2}}\rangle$ as
$L_{ds}\langle b_{i_{1}},b_{i_{2}}\rangle\textstyle\coprod L\langle u\rangle$
for $u=[b_{i_{1}},b_{i_{2}}]$ as in Remark~\ref{Bubremark}, and by
introducing the ideal $J$ discussed in the Introduction.
\end{remark} 

We begin by calculating $\qhlgy{\Omega\djks}$ using the Adams-Hilton
model \[AH(\djks)=\colim_{\sigma\in MF(K)} AH(FW(\sigma))\] in
Proposition~\ref{MdjksAH}. Let $b_{1},\ldots,b_{n}$ be the
generators in $AH(\djks)$ corresponding to the cells
$S^{m_{1}+1},\ldots,S^{m_{n}+1}$ respectively. For
$\sigma=(i_{1},\ldots,i_{k})\in K$, observe that
$\{b_{i_{1}},\ldots,b_{i_{k}}\}$ corresponds to the cells
$S^{m_{i}+1}$ which are in $FW(\sigma)$. Let
$N_{\sigma}=(\Sigma_{j=1}^{k} m_{i_{j}}+1)-2$. By
Theorem~\ref{Bubenik}, we have $\hlgy{AH(FW(\sigma))}\cong
   U(L_{ds}\langle b_{i_{1}},\ldots,b_{i_{k}}\rangle\coprod L\langle u_{\sigma}\rangle)$
where~$u_{\sigma}$ is the Hurewicz image of the adjoint of a higher
Whitehead product
\(\namedright{S^{N_{\sigma}+1}}{}{FW(\sigma)}\). 

We now prove Theorem~\ref{ULcolim}, restated below to match the 
simplifying assumption in Remark~\ref{simplifyremark}. 

\begin{theorem} 
   \label{ULcolim2} 
   Let $K$ be a directed $MF$-complex such that $\vert\sigma\vert>1$ 
   for every $\sigma\in MF(K)$. There is an algebra isomorphism
   \[\qhlgy{\Omega\djks}\cong U(L_{ds}\langle b_{1},\ldots,b_{n}\rangle
         \textstyle\coprod L\langle u_{\sigma}\mid\sigma\in MF(K)\rangle)\]
   where each $u_{\sigma}$ is the Hurewicz image of the adjoint of a
   higher Whitehead product. Further, the loop map 
   \(\namedright{\Omega\djks}{}{\prod_{i=1}^{n}\Omega S^{m_{i}+1}}\)
   is modelled by the map
   \[\namedright{U(L_{ds}\langle b_{1},\ldots,b_{n}\rangle\textstyle\coprod
        L\langle u_{\sigma}\mid\sigma\in MF(K)\rangle)/J}
       {U(\pi)}{UL_{ds}\langle b_{1},\ldots,b_{n}\rangle}\]
   where $\pi$ is the projection.
\end{theorem}

\begin{proof} 
Consider the string of isomorphisms
\[\begin{array}{lcl}
     \qhlgy{\Omega\djks} & \cong & \hlgy{AH(\djks} \\
        & \cong & \colim_{\sigma\in MF(K)}\hlgy{AH(FW(\sigma))} \\
        & \cong & \colim_{\sigma\in MF(K)}
             U(L_{ds}\langle b_{i_{1}},\ldots,b_{i_{k}}\rangle\coprod
                 L\langle u_{\sigma}\rangle) \\
        & \cong & U(\colim_{\sigma\in MF(K)}
             L_{ds}\langle b_{i_{1}},\ldots,b_{i_{k}}\rangle\coprod
                 L\langle u_{\sigma}\rangle) \\
        & \cong & U(L_{ds}\langle b_{1},\ldots,b_{n}\rangle\coprod
             L\langle u_{\sigma}\mid\sigma\in MF(K)\rangle).
  \end{array}\]
The first isomorphism holds because $AH(\djks)$ is an
Adams-Hilton model. The second isomorphism holds because
by Proposition~\ref{MdjksAH}. The third isomorphism holds by 
Theorem~\ref{Bubenik}. For the fourth isomorphism, Remark~\ref{Alldayremark} 
implies that in the case of directed $MF$-complexes (when the missing 
faces are glued along common faces, which topologically correspond 
to sub-products in \djks) the calculation of 
$\hlgy{AH(FW(\sigma))}\cong U(L_{ds}\langle b_{i_{1}},\ldots,b_{i_{k}}\rangle)$
is compatible with the inclusion of sub-products. Therefore both
the underlying Lie algebra and its universal enveloping algebra
respect the colimit over $MF(K)$, and so the fourth isomorphism holds. 
Provided the fifth isomorphism holds, the string of isomorphisms establishes 
the ismorphism asserted by the theorem. 
The statement regarding Hurewicz images now follows from that 
in Theorem~\ref{Bubenik}. The statement regarding the model for the looped map 
\(\namedright{\Omega\djks}{}{\prod_{i=1}^{n}\Omega S^{m_{i}+1}}\) 
follows again from Remark~\ref{Alldayremark} regarding the compatibility
of the colimit with the inclusion of sub-products.

It remains to establish the fifth isomorphism in the string of isomorphisms 
above. This is really a statement about Lie algebras, so we will show that 
\[\mbox{colim}_{\sigma\in MF(K)} L_{ds}\langle b_{i_{1}},\ldots,b_{i_{k}}\rangle 
    \coprod L\langle u_{\sigma}\rangle\cong 
   L_{ds}\langle b_{1},\ldots,b_{n}\rangle\coprod L\langle u_{\sigma}\vert 
      \sigma\in MF(K)\rangle.\] 
Since the set $MF(K)$ of minimal missing faces of $K$ is finite, the colimit 
can be rewritten as a free coproduct modulo relations. To indicate the 
dependence on $\sigma=(i_{1},\ldots,i_{k})$, write 
$L_{ds}\langle b_{i_{1}},\ldots,b_{i_{k}}\rangle\coprod L\langle u_{\sigma}\rangle$ 
as 
$L_{ds}\langle b^{\sigma}_{i_{1}},\ldots,b^{\sigma}_{i_{k}}\rangle 
     \coprod L\langle u_{\sigma}\rangle$ 
Rearranging terms, there is an isomorphism of free coproducts 
\[\coprod_{\sigma\in MF(K)} 
    \left(L_{ds}\langle b^{\sigma}_{i_{1}},\ldots,b^{\sigma}_{i_{k}}\rangle 
    \coprod L\langle u_{\sigma}\rangle\right)\cong 
   \left(\coprod_{\sigma\in MF(K)} 
     L_{ds}\langle b^{\sigma}_{i_{1}},\ldots,b^{\sigma}_{i_{k}}\rangle\right) 
     \coprod L\langle u_{\sigma}\vert\sigma\in MF(K)\rangle.\]
Thus there is an isomorphism  
\begin{multline} 
  \label{Liecolim} 
  \mbox{colim}_{\sigma\in MF(K)} 
    L_{ds}\langle b^{\sigma}_{i_{1}},\ldots,b^{\sigma}_{i_{k}}\rangle 
    \coprod L\langle u_{\sigma}\rangle\cong \\ 
    \left(\left(\coprod_{\sigma\in MF(K)} 
    L_{ds}\langle b^{\sigma}_{i_{1}},\ldots,b^{\sigma}_{i_{k}}\rangle\right) 
     \coprod L\langle u_{\sigma}\vert\sigma\in MF(K)\rangle\right)/\sim\hspace{1cm}  
\end{multline}  
where the relations $\sim$ are as follows. By definition of a directed 
$MF$-complex, two minimal missing faces $\sigma,\sigma^{\prime}\in MF(K)$ 
intersect along a common proper face, which we label as $\tau$. If 
$\tau=\emptyset$ then, as subspaces of $DJ_{K}(\underline{S})$, 
$FW(\sigma)$ and $FW(\sigma^{\prime})$ intersect only at the basepoint. 
The Adams-Hilton model of $FW(\sigma)\vee FW(\sigma^{\prime})$ is 
therefore the coproduct of the Adams-Hilton models for each individual 
space, so after taking homology no relation is introduced on the underlying 
Lie algebras in~(\ref{Liecolim}). If $\tau\neq\emptyset$, suppose that 
$\tau=(r_{1},\ldots,r_{s})$. Topologically, this intersection corresponds to the 
inclusion of the proper subproduct $\prod_{t=1}^{s} S^{2}$ into  
$DJ_{K}(\underline{S})$. The Adams-Hilton model for the loop space of this subproduct 
has homology isomorphic to $UL\langle b^{\tau}_{r_{1}},\ldots,b^{\tau}_{r_{s}}\rangle$, 
and by Remark~\ref{Alldayremark} the inclusion of this subproduct is compatible with 
the Adams-Hilton models for both $\Omega FW(\sigma)$ and 
$\Omega FW(\sigma^{\prime})$. Therefore, upon taking homology, in 
$U(L_{ds}\langle b^{\sigma}_{i_{1}},\ldots,b^{\sigma}_{i_{k}}\rangle 
     \coprod L\langle u_{\sigma}\rangle)$ 
and 
$U(L_{ds}\langle b^{\sigma^{\prime}}_{i_{1}},\ldots,b^{\sigma^{\prime}}_{i_{k^{\prime}}}\rangle 
     \coprod L\langle u_{\sigma^{\prime}}\rangle)$ 
we have $b^{\sigma}_{r_{t}}=b^{\sigma^{\prime}}_{r_{t}}$ for every 
$r_{1},\ldots,r_{t}$. This relation holds equally well on the underlying Lie 
algebras. Hence the relation $\sim$ in~(\ref{Liecolim}) is given by identifying 
$b^{\sigma}_{r_{\ell}}$ and $b^{\sigma^{\prime}}_{r_{\ell}}$ whenever 
$r_{\ell}\in\sigma\cap\sigma^{\prime}$. Consequently, considering all 
the minimal missing faces of $K$, we obtain 
\begin{multline*} 
  \left(\left(\coprod_{\sigma\in MF(K)} 
    L_{ds}\langle b^{\sigma}_{i_{1}},\ldots,b^{\sigma}_{i_{k}}\rangle\right) 
     \coprod L\langle u_{\sigma}\vert\sigma\in MF(K)\rangle\right)/\sim\ \cong \\ 
    L_{ds}\langle b_{1},\ldots,b_{n}\rangle\coprod L\langle u_{\sigma}\vert 
      \sigma\in MF(K)\rangle.\hspace{1cm} 
\end{multline*}  
\end{proof} 

Theorem~\ref{ULcolim2} is the crucial algebraic result. We first use
it to determine $\qhlgy{\Omega\zks}$, and then to determine a more
detailed description of the Hurewicz homomorphism.

Since
$L_{ds}\langle b_{1},\ldots,b_{n}\rangle\coprod
    L\langle u_{\sigma}\mid\sigma\in MF(K)\rangle$
is a coproduct, there is a short exact sequence of graded Lie algebras
\begin{equation}
  \label{SESLie}
   \nameddright{L\langle R\rangle}{i}
    {L_{ds}\langle b_{1},\ldots,b_{n}\rangle\textstyle\coprod
     L\langle u_{\sigma}\mid\sigma\in MF(K)\rangle}
    {\pi}{L_{ds}\langle b_{1},\ldots,b_{n}\rangle}
\end{equation}
where $i$ is the inclusion, $\pi$ is the projection, and
\begin{equation} 
  \label{Rindex} 
  R=\{[[u_{\sigma},b_{j_{1}}],\ldots, b_{j_{l}}]\mid\sigma\in MF(K),
     1\leq j_{1}\leq\cdots\leq j_{l}\leq n, 0\leq l<\infty\}. 
\end{equation} 
Here, when $l=0$ we interpret the bracket as simply being $u_{\sigma}$.

\begin{proposition}
   \label{hlgyloopzks}
   There is a commutative diagram of algebras
   \[\diagram
        \qhlgy{\Omega\zks}\rto^-{(\Omega f)_{\ast}}\dto^{\cong}
           & \qhlgy{\Omega\djks}\dto^{\cong} \\
        UL\langle R\rangle\rto^-{U(i)}
           & U(L_{ds}\langle b_{1},\ldots,b_{n}\rangle\coprod
             L\langle u_{\sigma}\mid\sigma\in MF(K)\rangle).
     \enddiagram\]
\end{proposition}

\begin{proof}
By~\cite[3.7]{CMN}, a short exact sequence of graded Lie algebras
induces a short exact sequence of Hopf algebras. In our case,
(\ref{SESLie})~induces a short exact sequence
\[\nameddright{UL\langle R\rangle}{U(i)}
    {U(L_{ds}\langle b_{1},\ldots,b_{n}\rangle\textstyle\coprod
     L\langle u_{\sigma}\mid\sigma\in MF(K)\rangle)}
    {U(\pi)}{UL_{ds}\langle b_{1},\ldots,b_{n}\rangle}.\]
Here, by a short exact sequence of Hopf algebras, we mean that there
is an isomorphism
\[U(L_{ds}\langle b_{1},\ldots,b_{n}\rangle\textstyle\coprod
     L\langle u_{\sigma}\mid\sigma\in MF(K)\rangle)\cong
     UL_{ds}\langle b_{1},\ldots,b_{n}\rangle\otimes UL\langle R\rangle\]
as right $UL\langle R\rangle$-modules and left
$UL_{ds}\langle b_{1},\ldots,b_{n}\rangle$-comodules.
In particular, $U(i)$ is the algebra kernel of $U(\pi)$. On the other hand,
Theorem~\ref{ULcolim2} implies that $U(\pi)$ is a model for the looped
inclusion
\(\namedright{\Omega\djks}{\Omega g}{\prod_{i=1}^{n}\Omega S^{m_{i}+1}}\).
Since $\Omega g$ has a right homotopy inverse, the homotopy decomposition
$\Omega\djks\simeq
    (\prod_{i=1}^{n}\Omega S^{m_{i}+1})\times\Omega\zks$
implies that there is a short exact sequence of Hopf algebras
\(\nameddright{\qhlgy{\Omega\zks}}{}{\qhlgy{\Omega\djks}}
    {U(\pi)}{\qhlgy{\prod_{i=1}^{n}\Omega S^{m_{i}+1}}}\).
Thus $\qhlgy{\Omega\zks}$ is also the algebra kernel of $U(\pi)$,
and the proposition follows.
\end{proof}

Theorem~\ref{ULcolim2} implies that the rational homology of $\Omega\djks$
is generated by Hurewicz images. To be precise, suppose that 
$\sigma=(i_{1},\ldots,i_{k})\in MF(K)$. Recall that 
$\underline{S}=(S^{m_{i_{1}}+1},\ldots,S^{m_{i_{k}}+1})$, and let 
$N_{\sigma}=(\Sigma_{j=1}^{k} m_{i_{j}+1})-2$. Let 
\[w_{\sigma}\colon\nameddright{S^{N_{\sigma}+1}}{\phi_{k}}{FW(\sigma)}
       {}{\djks}\]
be the higher Whitehead product. Let
\[s_{\sigma}\colon\namedright{S^{N_{\sigma}}}{}{\Omega\djks}\]
be the adjoint of $w_{\sigma}$. As stated in Theorem~\ref{ULcolim2},
the element $u_{\sigma}\in\qhlgy{\Omega\djks}$ is the Hurewicz
image of~$s_{\sigma}$. For $1\leq i\leq n$, let
\[a_{i}\colon\namedright{S^{m_{i}+1}}{}{\djks}\]
be the coordinate inclusion and let
\[\bar{a}_{i}\colon\namedright{S^{m_{i}}}{}{\Omega\djks}\]
be the adjoint of $a_{i}$. The Hurewicz image of $\bar{a}_{i}$ is $b_{i}$.
Let $\mathcal{I}$ be the index set for $R$ introduced in~(\ref{Rindex}). 
Then $\alpha\in\mathcal{I}$
corresponds to a face $\sigma\in MF(K)$ and a sequence
$(j_{1},\ldots,j_{l})$ where $1\leq j_{1}\leq\cdots\leq j_{l}\leq n$
and $0\leq l<\infty$. Given such an $\alpha$, let
$t_{\alpha}=(\Sigma_{t=1}^{l} m_{j_{t}})+N_{\sigma}$. Then there
is a Samelson product
\[[[s_{\sigma},\bar{a}_{j_{1}}],\ldots,\bar{a}_{j_{l}}]\colon
      \namedright{S^{t_{\alpha}}}{}{\Omega\djks}.\]
Since Samelson products commute with Hurewicz images, the
Hurewicz image of $[[s_{\sigma},\bar{a}_{j_{1}}],\ldots,\bar{a}_{j_{l}}]$
is $[[u_{\sigma},b_{j_{1}}],\ldots,b_{j_{l}}]$.
Adjointing, we have a Whitehead product
\[[[w_{\sigma},a_{j_{1}}],\ldots,a_{j_{l}}]\colon
      \namedright{S^{t_{\alpha}+1}}{}{\djks}.\]
Taking the wedge sum over all possible $\alpha$, we obtain a map
\[W\colon\namedright{\bigvee_{\alpha\in\mathcal{I}} S^{t_{\alpha}+1}}{}
    {\djks}.\]

\begin{corollary}
   \label{Hurewiczzks}
   The map
   \(\namedright{\Omega(\bigvee_{\alpha\in\mathcal{I}} S^{t_{\alpha}+1})}
        {\Omega W}{\Omega\djks}\)
   induces in rational homology the map
   \(\namedright{UL\langle R\rangle}{U(i)}
        {U(L_{ds}\langle b_{1},\ldots,b_{n}\rangle\coprod
             L\langle u_{\sigma}\mid\sigma\in MF(K)\rangle)}\).
\end{corollary}

\begin{proof}
Let $S$ be the composite
\[S\colon\nameddright{\bigvee_{\alpha\in\mathcal{I}} S^{t_{\alpha}}}{E}
    {\Omega(\bigvee_{\alpha\in\mathcal{I}} S^{t_{\alpha}+1})}{\Omega W}
    {\Omega\djks}.\]
Then $S$ is homotopic to the adjoint of $W$. In particular, the
wedge summands of $S$ are the Samelson products
$[[s_{\sigma},\bar{a}_{j_{1}}],\ldots,\bar{a}_{j_{l}}]$ for
$\alpha\in\mathcal{I}$. Thus, taking Hurewicz images, $S_{\ast}$
is the composite
\[\bar{i}\colon R\hookrightarrow\nameddright{UL\langle R\rangle}{U(i)}
    {U(L_{ds}\langle b_{1},\ldots,b_{n}\rangle\coprod
             L\langle u_{\sigma}\mid\sigma\in MF(K)\rangle)}{\cong}
    {\hlgy{\Omega\djks}}.\]
By the Bott-Samelson Theorem,
$\qhlgy{\Omega(\bigvee_{\alpha\in\mathcal{I}} S^{t_{\alpha}+1})}\cong
    T(\rqhlgy{\bigvee_{\alpha\in\mathcal{I}} S^{t_{\alpha}}})$,
and the latter algebra is isomorphic to $UL\langle R\rangle$.
Therefore, as $(\Omega W)_{\ast}$ is the multiplicative extension
of $S_{\ast}$, it induces the multiplicative extension $U(i)$ of $\bar{i}$.
\end{proof}

Finally, we bring \zks\ back into the picture.

\begin{theorem}
   \label{zksdecomp}
   The map
   \(\namedright{\bigvee_{\alpha\in\mathcal{I}} S^{t_{\alpha}+1}}{W}{\djks}\)
   lifts to \zks, and induces a homotopy equivalence
   \(\namedright{\bigvee_{\alpha\in\mathcal{I}} S^{t_{\alpha}+1}}{}{\zks}\).
\end{theorem}

\begin{proof}
By Lemma~\ref{fgsigmasplit}, each higher Whitehead product $w_{\sigma}$
lifts to \zks. Therefore each iterated Whitehead product
$[[w_{\sigma},a_{j_{1}}],\ldots,a_{j_{l}}]$ into \djks\ composes
trivially to $\prod_{i=1}^{n} S^{m_{i}+1}$ and so lifts to \zks. Hence
there is a lift
\begin{equation} 
  \label{lambdalift} 
  \diagram
     & \bigvee_{\alpha\in\mathcal{I}} S^{t_{\alpha}+1}\dto^{W}\dlto_-{\lambda} \\
     \zks\rto & \djks
  \enddiagram 
\end{equation} 
for some map $\lambda$.

After looping, Corollary~\ref{Hurewiczzks} implies that the map
\(\namedright{\Omega(\bigvee_{\alpha\in\mathcal{I}} S^{t_{\alpha}+1})}
        {\Omega\lambda}{\Omega\zks}\)
induces an inclusion
\(\namedright{UL\langle R\rangle}{(\Omega\lambda)_{\ast}}
    {\qhlgy{\Omega\zks}}\).
By Proposition~\ref{hlgyloopzks}, there is an isomorphism
$\qhlgy{\Omega\zks}\cong UL\langle R\rangle$. Therefore a counting
argument implies that the inclusion $(\Omega\lambda)_{\ast}$ must be
an isomorphism. Hence $\Omega\lambda$ is a rational homotopy
equivalence. That is, $\Omega\lambda$ induces an isomorphism of
rational homotopy groups. Therefore, so does $\lambda$, and so
$\lambda$ is a rational homotopy equivalence.

To upgrade this to an integral homotopy equivalence, proceed as follows. 

\noindent 
\textbf{Step 1}: Let 
\(g_{\sigma}\colon\namedright{S^{N_{\sigma}+1}}{}{\zks}\) 
be a lift of the higher Whitehead product $w_{\sigma}$. By 
Lemma~\ref{fgsigmasplit}, $g_{\sigma}$ can be chosen so that 
it has a left homotopy inverse. Now take the adjoint of~(\ref{lambdalift}) 
to obtain a homotopy commutative diagram 
\[\diagram 
      & \bigvee_{\alpha\in\mathcal{I}} S^{t_{\alpha}} 
               \dto^{\overline{W}}\dlto_{\overline{\lambda}} \\ 
       \Omega\mathcal{Z}_{K}(\overline{S})\rto^-{\Omega\varphi} 
           & \Omega DJ_{K}(\underline{S}) 
  \enddiagram\]
where $\overline{W},\overline{\lambda}$ are the adjoints of $W,\lambda$ 
respectively. Let 
\(\bar{g}_{\sigma}\colon\namedright{S^{N_{\sigma}}}{} 
    {\Omega\mathcal{Z}_{K}(\underline{S})}\) 
be the adjoint of $g_{\sigma}$. Then $\bar{g}_{\sigma}$ is homotopic to the composite 
\(\nameddright{S^{N_{\sigma}}}{E}{\Omega S^{N_{\sigma}+1}}{\Omega g_{\sigma}} 
    {\Omega\mathcal{Z}_{K}(\underline{S})}\) 
where $E$ is the suspension map. Since $g_{\sigma}$ has a left homotopy 
inverse, so does $\Omega g_{\sigma}$. Since $E_{\ast}$ is the inclusion 
of the generator in 
$\hlgy{\Omega S^{N_{\sigma}+1};\mathbb{Z}}\cong\mathbb{Z}[x_{N_{\sigma}}]$, 
we have that $\bar{g}_{\sigma}$ has nonzero Hurewicz image 
$\bar{u}_{\sigma}$ for some element 
$\bar{u}_{\sigma}\in H_{t}(\Omega\mathcal{Z}_{K}(\underline{S});\mathbb{Z})$.  
Let 
\(\bar{w}_{\sigma}\colon\namedright{S^{N_{\sigma}}}{}{\Omega DJ_{K}(\underline{S})}\) 
be the adjoint of $w_{\sigma}$. Since $g_{\sigma}$ is a lift of~$w_{\sigma}$ 
through $\varphi$, $\bar{g}_{\sigma}$ is a lift of~$\bar{w}_{\sigma}$ 
through $\Omega\varphi$. Therefore, if we identify 
$\bar{u}_{\sigma}\in\hlgy{\Omega\mathcal{Z}_{K}(\underline{S}):\mathbb{Z}}$ 
with its image in $\hlgy{\Omega DJ_{K}(\underline{S});\mathbb{Z}}$, 
then $\bar{w}_{\sigma}$ has Hurewicz image $\bar{u}_{\sigma}$. 

\noindent 
\textbf{Step 2}: 
Observe that the homotopy fibration 
\(\nameddright{\Omega\mathcal{Z}_{K}(\underline{S})}{\Omega\varphi} 
        {\Omega DJ_{K}(\underline{S})}{}{\prod_{i=1}^{n} S^{1}}\) 
splits as 
\[\Omega DJ_{K}(\underline{S})\simeq 
    \prod_{i=1}^{n} S^{1}\times\Omega\mathcal{Z}_{K}(\underline{S}).\] 
For  $1\leq i\leq n$, the coordinate inclusion 
\(\namedright{S^{2}}{a_{i}}{DJ_{K}(\underline{S})}\) 
adjoints to a map 
\(\namedright{S^{1}}{\bar{a}_{i}}{\Omega DJ_{K}(\underline{S})}\). 
The decomposition of $\Omega DJ_{K}(\underline{S})$ implies that 
$\bar{a}_{i}$ has a nonzero Hurewicz image 
$\bar{b}_{i}\in H_{1}(\Omega DJ_{K}(\underline{S});\mathbb{Z})$.  
As Samelson products commute with 
Hurewicz images, the iterated Samelson product 
$[[\bar{w}_{\sigma},\bar{a}_{j_{1}}],\ldots,\bar{a}_{j_{l}}]$ 
has Hurewicz image 
$[[\bar{u}_{\sigma},\bar{b}_{j_{1}}],\ldots,\bar{b}_{j_{l}}]$. 

\noindent 
\textbf{Step 3}: 
Now consider the rationalization map 
\[r\colon\namedright{\hlgy{\Omega DJ_{K}(\underline{S});\mathbb{Z}}} 
         {}{\hlgy{\Omega DJ_{K}(\underline{S});\mathbb{Q}}}.\] 
The integral splittings of spheres off $\mathcal{Z}_{K}(\underline{S})$ 
via the maps $g_{\sigma}$ and the integral decomposition of 
$\Omega DJ_{K}(\underline{S})$ induce corresponding rational splittings 
and a rational decomposition. Thus $r(\bar{u}_{\sigma})=u_{\sigma}$ 
and $r(\bar{b}_{i})=b_{i}$, where $u_{\sigma}$ and $b_{i}$ is the notation 
used for the rational classes in $\hlgy{\Omega DJ_{K}(\underline{S});\mathbb{Q}}$ 
in Corollary 6.4. As $r$ commutes with commutators, we therefore have  
$r([[\bar{u}_{\sigma},\bar{b}_{j_{1}}],\ldots,\bar{b}_{j_{l}}])= 
        [[u_{\sigma},b_{j_{1}}],\ldots, b_{j_{l}}]$. 
In particular, since $[[u_{\sigma},b_{j_{1}}],\ldots,b_{j_{l}}]$ is a 
generator of $R$ (notation as in Corollary~6.4), the Hurewicz image 
$[[\bar{u}_{\sigma},\bar{b}_{j_{1}}],\ldots,\bar{b}_{j_{l}}]$ in 
$\hlgy{\Omega DJ_{K}(\underline{S});\mathbb{Z}}$ is nonzero 
and maps by a degree one map to $[[u_{\sigma},b_{j_{1}}],\ldots b_{j_{l}}]$. 

Notice that the iterated Samelson product 
$[[\bar{w}_{\sigma},\bar{a}_{j_{1}}],\ldots,\bar{a}_{j_{l}}]$ 
is the adjoint of the iterated Whitehead product 
$[[w_{\sigma},a_{j_{1}}],\ldots,a_{j_{l}}]$. Therefore the collection 
of adjointed higher Whitehead products~$\bar{w}_{\sigma}$ and 
iterated Samelson products $[[\bar{w}_{\sigma},\bar{a}_{j_{1}}],\ldots,\bar{a}_{j_{l}}]$ 
comprise the map $\overline{W}$, and collectively lift by the 
map $\overline{\lambda}$ to $\Omega\mathcal{Z}_{K}(\underline{S})$.   
Each has an integral Hurewicz image which the rationalization map $r$ sends by 
a degree one map to a generator in 
$R\subseteq\hlgy{\Omega\mathcal{Z}_{K}(\underline{S});\mathbb{Q}} 
     \subseteq\hlgy{\Omega DJ_{K}(\underline{S});\mathbb{Q}}$. 
Moreover, these integral Hurewicz images are linearly independent since their 
rationalizations are, and they are in one-to-one correspondence 
with the rational classes in $R$.  

\noindent 
\textbf{Step 4}: 
By Theorem 1.4, $\mathcal{Z}_{K}(\underline{S})$ is homotopy 
equivalent to a wedge of simply-connected spheres, say 
$\mathcal{Z}_{K}(\underline{S})\simeq 
    \Sigma(\bigvee_{\beta\in\mathcal{J}} S^{t_{\beta}})$. 
Thus $R\cong\hlgy{\bigvee_{\beta\in\mathcal{J}} S^{t_{\beta}};\mathbb{Q}}$ 
and the inclusion of $R$ into 
$\hlgy{\Omega\mathcal{Z}_{K}(\underline{S});\mathbb{Q}}$ is 
induced by the suspension map $E$. Therefore, $\overline{\lambda}$ 
maps $\hlgy{\bigvee_{\alpha\in\mathcal{I}} S^{t_{\alpha}};\mathcal{Z}}$ 
isomorphically onto a set 
$R^{\prime}\subseteq\hlgy{\bigvee_{\beta\in\mathcal{J}} S^{t_{\beta}};\mathbb{Z}}$,  
which the rationalization $r$ sends by degree one maps to $R$. But $r$ 
sends $\hlgy{\bigvee_{\beta\in\mathcal{J}} S^{t_{\beta}};\mathbb{Z}}$ by 
degree one maps to $\hlgy{\bigvee_{\beta\in\mathcal{J}} S^{t_{\beta}};\mathbb{Q}}$.  
Thus $R^{\prime}\cong\hlgy{\bigvee_{\beta\in\mathcal{J}} S^{t_{\beta}};\mathbb{Z}}$. 
Adjointing, we obtain that the map 
\(\namedright{\bigvee_{\alpha\in\mathcal{I}} S^{t_{\alpha}+1}}{\lambda} 
     {\mathcal{Z}_{K}(\underline{S})}\) 
induces an isomorphism in integral homology. Hence $\lambda$ is 
an integral homotopy equivalence. 
\end{proof}

\noindent
\textit{Proof of Theorem~\ref{djksmaps}}:
This is simply a rephrasing of Theorem~\ref{zksdecomp}.
$\qqed$

\section{Properties of $\Omega\djk$ and $\Omega\zk$ for directed  
     $MF$-complexes}
\label{sec:djkproperties}

Recall that 
\(\imath\colon\namedright{S^{2}}{}{\cpinf}\)
is the inclusion of the bottom cell. By naturality, there is a
homotopy fibration diagram
\[\diagram
    \zkstwo\rto\dto^{\zk(\imath)}
       & \djkstwo\rto\dto^{\djk(\imath)}
       & \prod_{i=1}^{r} S^{2}\dto^{\prod_{i=1}^{r}\imath} \\
    \zk\rto & \djk\rto & \prod_{i=1}^{r}\cpinf.
  \enddiagram\]
In this section we will use the calculations of $\qhlgy{\Omega\djkstwo}$
and $\qhlgy{\Omega\zkstwo}$ in Section~\ref{sec:djksproperties} to
calculate $\qhlgy{\Omega\djk}$, proving
Theorem~\ref{hlgyloopdjk}, and $\qhlgy{\Omega\zk}$. We then give a
homotopy decomposition of \zk\ as a wedge of spheres and describe the map
\(\namedright{\zk}{}{\djk}\)
in terms of higher Whitehead products and iterated Whitehead products,
proving Theorem~\ref{djkmaps}.

By Proposition~\ref{MdjkAH},
$\hlgy{\Omega\djk}\cong\colim_{\sigma\in MF(K)}\qhlgy{\Omega
FW(\sigma)}$, so we first need to calculate $\qhlgy{\Omega
FW(\sigma)}$ and then take a colimit to put the pieces together. We
do this in Lemma~\ref{djkBubenik} and Theorem~\ref{hlgyloopdjk2}
after two preliminary lemmas. In general, let $X_{1},\ldots,X_{n}$
be path-connected spaces and consider the fat wedge
$FW(\underline{X})$ in $\prod_{i=1}^{n} X_{i}$. Let $j$ be the
inclusion \(j\colon\namedright{FW(\underline{X})}{}{\prod_{i=1}^{n}
X_{i}}\).

\begin{lemma}
   \label{natFWinverse}
   The map
   \(\namedright{\Omega FW(\underline{X})}{\Omega j}
       {\prod_{i=1}^{n}\Omega X_{i}}\)
   has a right homotopy inverse, which can be chosen to be natural
   for maps
   \(\namedright{X_{i}}{}{Y_{i}}\).
\end{lemma}

\begin{proof}
The inclusion
\(\namedright{\bigvee_{i=1}^{n} X_{i}}{}{FW(\underline{X})}\)
is natural, as are the inclusions
\(\namedright{X_{i}}{}{\bigvee_{i=1}^{n} X_{i}}\)
for $1\leq i\leq n$. Now loop and consider the composite
\[m\colon\namedddright{\prod_{i=1}^{n}\Omega X_{i}}{}
   {\prod_{i=1}^{n}\Omega(\bigvee_{i=1}^{n} X_{i})}{\mu}
   {\Omega(\bigvee_{i=1}^{n} X_{i})}{}{\Omega FW(\underline{X})},\]
where $\mu$ is the loop multiplication. All three maps in the
composite are natural, and $m$ is a right homotopy inverse of $\Omega j$.
\end{proof}

Let $F$ be the homotopy fibre of $j$. As mentioned earlier,
Porter~\cite{P2} showed that there is a homotopy equivalence
$F\simeq\Sigma^{n-1}\Omega X_{1}\wedge\cdots\wedge\Omega X_{n}$.
Further, in~\cite[1.2]{P1} he showed that this homotopy equivalence is
natural for maps
\(\namedright{X_{i}}{}{Y_{i}}\).
We record this as follows.

\begin{lemma}
   \label{FWfibrenat}
   Let
   \(f_{i}\colon\namedright{X_{i}}{}{Y_{i}}\)
   be maps between simply-connected spaces. There is a homotopy
   commutative diagram between fibrations
   \[\diagram
        \Sigma^{n-1}\Omega X_{1}\wedge\cdots\wedge\Omega X_{n}\rrto
              \dto^{\Sigma^{n-1}\Omega f_{1}\wedge\cdots\wedge\Omega f_{n}}
           & & FW(\underline{X})\rrto^-{j}\dto^{FW(f_{1},\ldots,f_{n})}
           & & \prod_{i=1}^{n} X_{i}\dto^{\prod_{i=1}^{n} f_{i}} \\
        \Sigma^{n-1}\Omega Y_{1}\wedge\cdots\wedge\Omega Y_{n}\rrto
           & & FW(\underline{Y})\rrto^-{j} & & \prod_{i=1}^{n} X_{i}.
     \enddiagram\]
   $\qqed$
\end{lemma}

Let $FW(S^{2})$ and $FW(\cpinf)$ be the fat wedges of
$\prod_{i=1}^{n} S^{2}$ and $\prod_{i=1}^{n}\cpinf$ respectively. Let
\(FW(\imath)\colon\namedright{FW(S^{2})}{}{FW(\cpinf)}\)
be the map induced by the inclusion
\(\namedright{S^{2}}{\imath}{\cpinf}\).

\begin{lemma}
   \label{djkBubenik}
   There is a commutative diagram of algebras
   \[\diagram
         \qhlgy{\Omega FW(S^{2})}\rto^-{\cong}\dto^{(\Omega FW(\imath))_{\ast}}
            & U(L_{ds}\langle b_{1}\ldots,b_{n}\rangle\coprod L\langle u\rangle)
                \dto^{q} \\
         \qhlgy{\Omega FW(\cpinf)}\rto^-{\cong}
            & U(L_{ds}\langle b_{1},\ldots,b_{n}\rangle\coprod L\langle u\rangle)/I
     \enddiagram\]
   where $u$, of degree $2n-2$, is the Hurewicz image of the adjoint of
   a higher Whitehead product, $I$ is the ideal
   $(b_{i}^{2},[u,b_{i}]\mid 1\leq i\leq n)$, and $q$ is the quotient map.
\end{lemma}

\begin{proof}
The isomorphism for $\qhlgy{\Omega FW(S^{2})}$ holds by
Theorem~\ref{Bubenik}. To obtain the compatible isomorphism for
$\qhlgy{\Omega FW(\cpinf)}$ we first consider what happens on the
level of spaces. For a space $X$, let $X^{(n)}$ be the $n$-fold smash 
product of $X$ with itself. By Lemma~\ref{FWfibrenat}, the map $\imath$ 
induces a homotopy commutative diagram
\[\diagram
     \Sigma^{n-1}(\Omega S^{2})^{(n)}\rto\dto^{\Sigma^{n-1}(\Omega\imath)^{(n)}}
        & FW(S^{2})\rto^-{j}\dto^{FW(\imath)}
        & \prod_{i=1}^{n} S^{2}\dto^{\prod_{i=1}^{n}\imath} \\
     \Sigma^{n-1}(\Omega\cpinf)^{(n)}\rto
        & FW(\cpinf)\rto^-{j} & \prod_{i=1}^{n}\cpinf.
  \enddiagram\]
Note that $\Omega\cpinf\simeq S^{1}$ so
$\Sigma^{n-1}(\Omega\cpinf)^{(n)}\simeq S^{2n-1}$.
Also, since
\(\namedright{S^{1}}{E}{\Omega S^{2}}\)
is a right homotopy inverse for $\Omega\imath$, if we let
$s=\Sigma^{n-1} E^{(n)}$ and $t=\Sigma^{n-1}(\Omega i)^{(n)}$,
then the composite
\(\nameddright{S^{2n-1}}{s}{\Sigma^{n-1}(\Omega S^{2})^{(n)}}
    {t}{S^{2n-1}}\)
is homotopic to the identity map.

After looping we obtain a homotopy commutative diagram
\begin{equation}
  \label{FWdgrm}
  \diagram
     \Omega(\Sigma^{n-1}(\Omega S^{2})^{(n)})\rto\dto^{\Omega t}
        & \Omega FW(S^{2})\rto^-{\Omega j}\dto^{\Omega FW(\imath)}
        & \prod_{i=1}^{n} \Omega S^{2}\dto^{\prod_{i=1}^{n}\Omega\imath} \\
     \Omega S^{2n-1}\rto & \Omega FW(\cpinf)\rto^-{\Omega j}
        & \prod_{i=1}^{n} S^{1}.
  \enddiagram
\end{equation}
By Lemma~\ref{natFWinverse}, $\Omega j$ has a natural right homotopy
inverse, so there is a homotopy commutative diagram of sections
\begin{equation}
  \label{natsection}
  \diagram
     \prod_{i=1}^{n}\Omega S^{2}\rto^-{m}\dto^{\prod_{i=1}^{n}\Omega\imath}
         & \Omega FW(S^{2})\dto^{\Omega FW(\imath)} \\
     \prod_{i=1}^{n} S^{1}\rto^-{m} & \Omega FW(\cpinf).
  \enddiagram
\end{equation}

Now we examine the effect of~(\ref{natsection}) in homology.
By Theorem~\ref{ULcolim2} and Proposition~\ref{hlgyloopzks}, a
model for the homology of the homotopy fibration along the top
row of~(\ref{FWdgrm}) is
\begin{equation}
  \label{Bubmodel}
  \nameddright{UL\langle R\rangle}{U(i)}
    {U(L_{ds}\langle b_{1},\ldots,b_{n}\rangle\textstyle\coprod L\langle u\rangle)}
    {U(\pi)}{UL_{ds}\langle b_{1},\ldots,b_{n}\rangle} 
\end{equation}
where $R=\{[[u,b_{j_{1}}],\ldots, b_{j_{l}}]\mid
     1\leq j_{1}\leq\cdots\leq j_{l}\leq n, 0\leq l<\infty\}$.
From~(\ref{natsection}), we obtain a right inverse $m_{\ast}$ of $U(\pi)$.
In particular, if
$\qhlgy{\prod_{i=1}^{n}\Omega S^{2}}\cong\mathbb{Q}[c_{1},\ldots,c_{n}]$,
then $m_{\ast}(c_{i})=b_{i}+\gamma_{i}$ for some $\gamma_{i}\in UL\langle R\rangle$.
However, the least degree of $UL\langle R\rangle$ which is nontrivial
is $2n-2$, while $c_{i}$ has degree $1$. As $n\geq 3$, for degree reasons
we must have $\gamma_{i}=0$. Thus $m_{\ast}(c_{i})=b_{i}$. For similar
degree reasons, we have $m_{\ast}(c_{i}^{2})=b_{i}^{2}$ (even though $m_{\ast}$
may not be multiplicative). On the other hand, $(\Omega\imath)_{\ast}$
is an isomorphism on the first homology group and the same is true after
composing with $m_{\ast}$, while $(\Omega\imath)_{\ast}(c_{i}^{2})=0$.
Thus the commutativity of~(\ref{natsection}) after taking homology implies
that for $1\leq i\leq n$, $(\Omega FW(\imath))_{\ast}$ is degree one
on $b_{i}$ while $(\Omega FW(\imath))_{\ast}(b_{i}^{2})=0$.
The latter implies by multiplicativity that $(\Omega FW(\imath))_{\ast}$
sends the ideal $(b_{1}^{2},\ldots,b_{n}^{2})$ to $0$.

Next, consider the commutator
$[u,b_{i}]\in\qhlgy{\Omega FW(S^{2})}$. In terms
of~(\ref{Bubmodel}), $[u,b_{i}]$ composes trivially with $U(\pi)$
and so is the image of an element $\delta_{i}\in UL\langle R\rangle$.
Note that $\delta_{i}$ has degree $2n-1$. Taking homology in~(\ref{FWdgrm}),
we see that $(\Omega t)_{\ast}(\delta_{i})=0$ for degree reasons.
Thus the commutativity of the left square in~(\ref{FWdgrm}) implies
that $(\Omega FW(\imath))_{\ast}([u,b_{i}])=0$. By multiplicativity,
$(\Omega FW(\imath))_{\ast}$ therefore sends the ideal
$I=(b_{i}^{2},[u,b_{i}]\mid 1\leq i\leq n)$ to $0$.

Thus there is a factorization
\begin{equation}
  \label{factordgrm}
  \diagram
       \qhlgy{\Omega FW(S^{2})}\rto^-{\cong}\dto^{(\Omega FW(\imath))_{\ast}}
         & U(L_{ds}\langle b_{1}\ldots,b_{n}\rangle\coprod L\langle u\rangle)
             \dto^{q} \\
      \qhlgy{\Omega FW(\cpinf)}
         & U(L_{ds}\langle b_{1},\ldots,b_{n}\rangle\coprod L\langle u\rangle)/I
             \lto^-{h}
  \enddiagram
\end{equation}
for some algebra map $h$, which is degree one on $b_{i}$ for each
$1\leq i\leq n$. In addition, the fact that $\Omega t$ has a right
homotopy inverse implies that $h$ is degree one on $u$.

We claim that $h$ is an isomorphism, from which the lemma would follow.
To see the isomorphism, let
$I^{\prime}$ be the ideal $([u,b_{i}]\mid 1\leq i\leq n)$.
Observe that
$U(L_{ds}\langle b_{1},\ldots,b_{n}\rangle\coprod L\langle u\rangle)/I^{\prime}$
is isomorphic to
$UL_{ds}\langle b_{1},\ldots,b_{n},u\rangle\cong\mathbb{Q}[b_{1},\ldots b_{n},u]$.
Thus
$U(L_{ds}\langle b_{1},\ldots,b_{n}\rangle\coprod L\langle u\rangle)/I$ is
isomorphic to
$\Lambda(b_{1},\ldots,b_{n})\otimes\mathbb{Q}[u]$. On the other hand,
the section $m$ in~(\ref{natsection}) implies that there is a homotopy
decomposition
$\Omega FW(\cpinf)\simeq(\prod_{i=1}^{n} S^{1})\times\Omega S^{2n-1}$.
Thus there is a coalgebra isomorphism
$\qhlgy{\Omega FW(\cpinf)}\cong\Lambda(c_{1},\ldots,c_{n})\otimes\mathbb{Q}[v]$
where $v$ has degree $2n-2$. From the use of $m$ and $t$ in both the homotopy
decomposition of $\Omega FW(\cpinf)$ and the factorization of
$(\Omega FW(\imath))_{\ast}$ through $h$, we see that
$h(b_{i})=c_{i}$ for $1\leq i\leq n$ and $h(u)=v$. As~$h$ is an
algebra map, it therefore induces an isomorphism.
\end{proof}

Now we pass to a colimit of fat wedges to prove Theorem~\ref{hlgyloopdjk}, 
restated below to match the simplifying condition in Remark~\ref{simplifyremark}. 

\begin{theorem}
   \label{hlgyloopdjk2}
   Let $K$ be a directed $MF$-complex such that $\vert\sigma\vert>1$ 
   for every $\sigma\in MF(K)$. There is an algebra isomorphism
   \[\qhlgy{\Omega\djk}\cong
        U(L_{ds}\langle b_{1},\ldots,b_{n}\rangle\textstyle\coprod
        L\langle u_{\sigma}\mid\sigma\in MF(K)\rangle)/I\]
   where $u_{\sigma}$ is the Hurewicz image of the adjoint of a
   higher Whitehead product and $I$ is the ideal
   \[I=(b_{i}^{2},[u_{\sigma},b_{j_{\sigma}}]\mid
       1\leq i\leq n,\sigma=(i_{1},\ldots,i_{k})\in MF(K),
       j_{\sigma}\in\{i_{1},\ldots,i_{k}\}).\]
   Further, there is a commutative diagram of algebras
   \[\diagram
        \qhlgy{\Omega\djkstwo}\rto^-{\cong}\dto^{(\Omega\djk(\imath))_{\ast}}
            & U(L_{ds}\langle b_{1},\ldots,b_{n}\rangle\textstyle\coprod
                L\langle u_{\sigma}\mid\sigma\in MF(K)\rangle)\dto^{q} \\
        \qhlgy{\Omega\djk}\rto^-{\cong}
            & U(L_{ds}\langle b_{1},\ldots,b_{n}\rangle\textstyle\coprod
                L\langle u_{\sigma}\mid\sigma\in MF(K)\rangle)/I
     \enddiagram\]
   where $q$ is the quotient map.
\end{theorem}

\begin{proof} 
With $\sigma=(i_{1},\ldots,i_{k})$, consider the diagram
\[\spreaddiagramcolumns{-1pc}\diagram
      \qhlgy{\Omega\djkstwo}\rrto^-{(\Omega\djk(\imath))_{\ast}}\dto^{\cong}
         & & \qhlgy{\Omega\djk}\dto^-{\cong} \\
      \colim_{\sigma\in MF(K)}
             U(L_{ds}\langle b_{i_{1}},\ldots,b_{i_{k}}\rangle\coprod L\langle u_{\sigma}\rangle)
             \rrto^-{Q}\dto^{\cong}
         & & \colim_{\sigma\in MF(K)}
              U(L_{ds}\langle b_{i_{1}},\ldots,b_{i_{k}}\rangle\coprod
                L\langle u_{\sigma}\rangle)/I_{\sigma}\dto^-{\cong} \\
      U(L_{ds}\langle b_{1},\ldots,b_{n}\rangle\coprod
                L\langle u_{\sigma}\mid\sigma\in MF(K)\rangle)\rrto^-{q}
          & & U(L_{ds}\langle b_{i_{1}},\ldots,b_{i_{k}}\rangle\coprod
                L\langle u_{\sigma}\mid\sigma\in MF(K)\rangle)/I
  \enddiagram\]
where $Q=\colim_{\sigma\in MF(K)} q_{\sigma}$ and $I_{\sigma}$ is the ideal 
generated by $(b_{\sigma_{i}}^{2},[u_{\sigma},b_{j_{\sigma}}]\mid
    j_{\sigma}\in\{i_{1},\ldots,i_{k}\})$.
The upper square commutes by combining Proposition~\ref{MdjkAH} and
Lemma~\ref{djkBubenik}. The lower square is the result of evaluating
the colimit, and so commutes. Note that both squares commute as maps
of algebras. The lower row is the isomorphism asserted by the
theorem, and the outer rectangle is the asserted commutative
diagram. 
\end{proof} 

Next, we use Theorem~\ref{hlgyloopdjk2} to calculate $\qhlgy{\Omega\zk}$
in Proposition~\ref{hlgyloopzk} as the universal enveloping algebra of
a certain free graded Lie algebra. This will involve some explicit
calculations involving graded Lie algebra identities, which we recall
now. In general, if $L$ is a graded Lie algebra over $\mathbb{Q}$
with bracket $[\ ,\ ]$, there is a graded anti-symmetry identity
$[x,y]=-(-1)^{\vert x\vert\vert y\vert}[y,x]$
for all $x,y\in L$ and a graded Jacobi identity
$[[x,y],z]=[x,[y,z]]-(-1)^{\vert x\vert\vert y\vert}[y,[x,z]]$
for all $x,y,z\in L$.

The ideal in Theorem~\ref{hlgyloopdjk2} involves brackets of the
form $[u_{\sigma},b_{j}]$ where $j\in\{i_{1},\ldots,i_{k}\}$, where
$\sigma=(i_{1},\ldots,i_{k})$. Thus in the quotient we need to keep
track of brackets of the form $[u_{\sigma},b_{j}]$ where~$j$ is in
the complement of $\{i_{1},\ldots,i_{k}\}$. Let
$J_{\sigma}=\{1,\ldots,n\}-\{i_{1},\ldots,i_{k}\}$.
Consider the free graded Lie algebra generated by
\[\widetilde{R}=\{[[u_{\sigma},b_{j_{1}}],\ldots, b_{j_{l}}]\mid\sigma\in MF(K),
     \{j_{1},\ldots,j_{l}\}\subseteq J_{\sigma}, 1\leq j_{1}<\cdots< j_{l}\leq n,
     0\leq l\leq n\}.\]
Note that each $j_{t}$ can appear at most once in any given bracket. This should
be compared to the $\mathbb{Q}$-module $R$, where each $j_{t}$ can appear
arbitrarily many times in a given bracket. Let
\(i_{R}\colon\namedright{\widetilde{R}}{}{R}\)
be the inclusion and
\(\pi_{R}\colon\namedright{R}{}{\widetilde{R}}\)
be the projection.

\begin{lemma}
   \label{zkLie}
   There is a short exact sequence of Lie algebras
   \[\nameddright{L\langle\widetilde{R}\rangle}{\tilde{i}}
       {(L_{ds}\langle b_{1},\ldots,b_{n}\rangle\textstyle\coprod
        L\langle u_{\sigma}\mid\sigma\in MF(K)\rangle)/\widetilde{I}}
       {\widetilde{\pi}}{L_{ds}\langle b_{1},\ldots,b_{n}\rangle/\widetilde{I}^{\prime}}\]
   where $\widetilde{I}$ is the ideal
   $\widetilde{I}=([b_{i},b_{i}],[u_{\sigma},b_{j_{\sigma}}]\mid
       1\leq i\leq n,\sigma\in MF(K), j_{\sigma}\in\{i_{1},\ldots,i_{k}\})$,
   $\widetilde{I}^{\prime}$ is the ideal $([b_{i},b_{i}]\mid 1\leq i\leq n)$,
   $\tilde{i}$ is the inclusion, and $\widetilde{\pi}$ is the projection.
\end{lemma}

\begin{proof}
To simplify notation, let
$L=L_{ds}\langle b_{1},\ldots,b_{n}\rangle\textstyle\coprod
        L\langle u_{\sigma}\mid\sigma\in MF(K)\rangle$.
Observe from the definitions of $\widetilde{I}$ and $\widetilde{I}^{\prime}$
that there is a commutative diagram
\[\diagram
       L\rto^-{\pi}\dto^{q}
          & L_{ds}\langle b_{1},\ldots,b_{n}\rangle\dto^{q^{\prime}} \\
       L/\widetilde{I}\rto^-{\widetilde{\pi}}
          & L_{ds}\langle b_{1},\ldots,b_{n}\rangle/\widetilde{I}^{\prime}
  \enddiagram\]
where $q$ and $q^{\prime}$ are the quotient maps. By~(\ref{SESLie}), the
kernel of $\pi$ is $L\langle R\rangle$. Let $\widetilde{L}$ be the kernel
of $\widetilde{\pi}$. The commutativity of the diagram implies that there
is an induced map
\(\tilde{q}\colon\namedright{L\langle R\rangle}{}{\widetilde{L}}\).

We claim that $\tilde{q}$ is a surjection. Let $x\in\widetilde{L}$ and
let $x$ also denote its image in $L/\widetilde{I}$. As $q$ is onto
there is an element $y\in L$ such that $q(y)=x$. Let $z=\pi(y)$.
If $z=0$ then by exactness $y$ lifts to $\tilde{y}\in L\langle R\rangle$
and so $\tilde{q}(\tilde{y})=x$. If $z\neq 0$, then $q^{\prime}(z)=0$
by exactness. Since $L$ is a coproduct, the projection $\pi$ has a
right inverse
\(r\colon\namedright{L_{ds}\langle b_{1},\ldots,b_{n}\rangle}{}{L}\)
which is a map of Lie algebras. As the generators of the ideal
$\widetilde{I}^{\prime}$ are all generators of the ideal $\widetilde{I}$,
we have $(q\circ r)(m)=0$ if and only if $q^{\prime}(m)=0$ for any
$m\in L_{ds}\langle b_{1},\ldots,b_{n}\rangle$. Thus $r(z)$ has
the property that $(q\circ r)(z)=0$. Therefore $\tilde{y}=y-r(z)$
lifts to $L\langle R\rangle$ and $q(y-r(z))=q(y)=x$, so
$\tilde{q}(\tilde{y})=x$. Hence $\tilde{q}$ is a surjection.

Now $\tilde{q}$ is a surjection and $\widetilde{L}$ injects into
$L/\widetilde{I}$. Therefore $\widetilde{L}$ is isomorphic to
the image of $L\langle R\rangle$ under $q$. We next show that
this image is $L\langle\widetilde{R}\rangle$. We first perform
two short calculations.

\noindent\textit{Calculation 1}:
The Jacobi identity states that
$[[a,b_{i}],b_{j}]=[a,[b_{i},b_{j}]]-
   (-1)^{\vert a\vert\vert b_{i}\vert}[b_{i},[a,b_{j}]]$
for any $a\in L$ and any $1\leq i,j\leq n$. The abelian property of
$L_{ds}\langle b_{1},\ldots,b_{n}\rangle$ implies that $[b_{i},b_{j}]=0$
and so $[a,[b_{i},b_{j}]]=0$. By the anti-symmetry identity,
$-(-1)^{\vert a\vert\vert b_{i}\vert}[b_{i},[a,b_{j}]]=
    (-1)^{\vert a\vert\vert b_{i}\vert+\vert b_{i}\vert\vert[a,b_{j}]\vert}
    [[a,b_{j}],b_{i}]$.
Since $\vert b_{i}\vert=1$ for $1\leq i\leq n$, the sign on the right
side of this equation equals $(-1)^{2\vert a\vert+1}$, which is~$-1$.
Therefore $[[a,b_{i}],b_{j}]=-[[a,b_{j}],b_{i}]$.

\noindent\textit{Calculation 2}:
The Jacobi identity states that
$[[a,b_{i}],b_{i}]=[a,[b_{i},b_{i}]]
     -(-1)^{\vert a\vert\vert b_{i}\vert}[b_{i},[a,b_{i}]]$
for any $a\in L$ and $1\leq i\leq n$. Since $[b_{i},b_{i}]=0$
in $L$, we have $[a,[b_{i},b_{i}]]=0$. As in Calculation 1,
the anti-symmetry identity shows that
$-(-1)^{\vert u_{\sigma}\vert}[b_{i},[a,b_{i}]]=-[[a,b_{i}],b_{i}]$.
Thus $[[a,b_{i}],b_{i}]=-[[a,b_{i}],b_{i}]$,
and so $2[[a,b_{i}],b_{i}]=0$. As $L$ is a Lie algebra over
$\mathbb{Q}$, $2$ is invertible and so $[[a,b_{i}],b_{i}]=0$.

By Calculation 1, up to sign change, whenever consecutive
$b$'s appear in a bracket of $L\langle R\rangle$ or $L$ their order
can be interchanged. By Calculation 2, the effect of
taking the quotient in $L\langle R\rangle$ and $L$ by the ideal
$I^{\prime}=([b_{i},b_{i}]\mid 1\leq i\leq n)$
is to annihilate all brackets in which appears a copy of $[a,[b_{i},b_{i}]]$.
Together with Calculation~1 which lets us freely interchange
consecutive $b$'s, any bracket of
the form $[[u_{\sigma},b_{j_{1}}],\ldots, b_{j_{l}}]$ is zero if
any $b_{j_{t}}$ appears more than once. Thus, the only such nontrivial
brackets must have  $1\leq j_{1}<\cdots< j_{l}\leq n$, $0\leq l<n$, as
in the definition of $\widetilde{R}$. The effect of then taking the
quotient by the ideal generated by $[u_{\sigma},b_{j_{\sigma}}]$
for $j_{\sigma}\in\{i_{1},\ldots,i_{k}\}$ is to annihilate those brackets in
$\{[[u_{\sigma},b_{j_{1}}],\ldots, b_{j_{l}}]\mid\sigma\in MF(K),
    1\leq j_{1}<\cdots< j_{l}\leq n, 0\leq l<n\}$
which do not have $j_{1},\ldots,j_{l}\in J_{\sigma}$. Thus the
image of $L\langle R\rangle$ under $q$ is $L\langle\widetilde{R}\rangle$.
\end{proof}

In general, the image of a graded Lie algebra $L$ in its universal
enveloping algebra $UL$ has the property that
$[x,y]=xy-(-1)^{\vert x\vert\vert y\vert}yx$,
where the multiplication is taking place in $UL$. In particular,
the anti-symmetry identity implies that $[x,x]=2x^{2}$ if the
degree of $x$ is odd. Thus if $2$ has been inverted in the ground ring,
then the ideal in $UL$ generated by $[x,x]$ is identical to the
ideal generated by $x^{2}$. In our case, the short exact sequence
of Lie algebras in Lemma~\ref{zkLie} implies that there is a short
exact sequence of Hopf algebras
\begin{equation}
  \label{SESHopf}
  \nameddright{L\langle\widetilde{R}\rangle}{}
    {(L_{ds}\langle b_{1},\ldots,b_{n}\rangle\textstyle\coprod
     L\langle u_{\sigma}\mid\sigma\in MF(K)\rangle)/I}
    {\pi}{L_{ds}\langle b_{1},\ldots,b_{n}\rangle/I^{\prime}}
\end{equation}
where $I$ is the ideal in Theorem~\ref{hlgyloopdjk} and
$I^{\prime}=(b_{i}^{2}\mid 1\leq i\leq n)$.

\begin{proposition}
   \label{hlgyloopzk}
   There is a commutative diagram of algebras
   \[\diagram
        \qhlgy{\Omega\zk}\rto\dto^{\cong}
           & \qhlgy{\Omega\djk}\dto^{\cong} \\
        UL\langle\widetilde{R}\rangle\rto^-{U(\tilde{i})}
           & U(L_{ds}\langle b_{1},\ldots,b_{n}\rangle\coprod
             L\langle u_{\sigma}\mid\sigma\in MF(K)\rangle)/I.
     \enddiagram\]
\end{proposition}

\begin{proof}
Argue as in Proposition~\ref{hlgyloopzks}, replacing the short exact
sequence of Hopf algebras appearing there with that
in~\eqref{SESHopf}, and replacing Theorem~\ref{ULcolim2} with
Theorem~\ref{hlgyloopdjk2}.
\end{proof}

Now we use the description of $\qhlgy{\Omega\djk}$ in
Theorem~\ref{hlgyloopdjk2} to produce maps as was done in the case of $\Omega\djks$.
For $\sigma\in MF(K)$, let
\[\tilde{w}_{\sigma}\colon\namedddright{S^{2\vert\sigma\vert+1}}{\phi_{k}}
       {FW(S^{2},\sigma)}{FW(\imath)}{FW(\sigma)}{}{\djk}\]
be the higher Whitehead product. By Theorem~\ref{hlgyloopdjk2}, the element
$u_{\sigma}\in\qhlgy{\Omega\djks}$ is the Hurewicz image of
the adjoint of $\tilde{w}_{\sigma}$. For $1\leq i\leq n$,
let $\tilde{a}_{i}$ be the composite
\[\tilde{a}_{i}\colon\nameddright{S^{2}}{\imath}{\cpinf}{}{\djk}\]
where the right map is the $i^{th}$-coordinate inclusion. Let
$\widetilde{\mathcal{I}}$ be the index set for $\widetilde{R}$. Then
$\widetilde{\alpha}\in\widetilde{\mathcal{I}}$ corresponds to a face
$\sigma\in MF(K)$ and a sequence $(j_{1},\ldots,j_{l})$ where $1\leq
j_{1}<\cdots< j_{l}\leq n$ and $0\leq l\leq n$. Given such an
$\widetilde{\alpha}$, let $t_{\widetilde{\alpha}}=N_{\sigma}+l-2$.
The inclusion $i_{R}$ induces a map
\(\widetilde{i}_{R}\colon\namedright{\bigvee_{\widetilde{\alpha}\in\widetilde{\mathcal{I}}}
    S^{t_{\widetilde{\alpha}}+1}}{}{\bigvee_{\alpha\in\mathcal{I}} S^{t_{\alpha}+1}}\).
Note that $(\Omega\widetilde{i}_{R})_{\ast}$ can be identified with $U(i_{R})$.
Consider the composite
\[\widetilde{W}\colon\namedddright
    {\bigvee_{\widetilde{\alpha}\in\widetilde{\mathcal{I}}}S^{t_{\widetilde{\alpha}}+1}}
    {\widetilde{i}_{R}}{\bigvee_{\alpha\in\mathcal{I}} S^{t_{\alpha}+1}}{W}{\djkstwo}
    {\djk(\imath)}{\djk}.\]
If $\widetilde{\alpha}$ indexes $\sigma\in MF(K)$, then the
restriction of $\widetilde{W}$ to $S^{t_{\widetilde{\alpha}}+1}$ is the
higher Whitehead product~$\widetilde{w}_{\sigma}$. Otherwise, the
restriction of $\widetilde{W}$ to $S^{t_{\widetilde{\alpha}}+1}$ is an
iterated Whitehead product of a single $\widetilde{w}_{\sigma}$ with some
selection of the coordinate inclusions $\tilde{a}_{1},\ldots,\tilde{a}_{n}$,
where each $\tilde{a}_{i}$ appears at most once.

\begin{corollary}
   \label{Hurewiczzk}
   The map
   \(\namedright{\Omega(\bigvee_{\widetilde{\alpha}\in\widetilde{\mathcal{I}}}
        S^{t_{\widetilde{\alpha}}+1})}{\Omega\widetilde{W}}{\Omega\djk}\)
   induces in homology the map
   \(\namedright{UL\langle\widetilde{R}\rangle}{U(\tilde{i})}
        {U(L_{ds}\langle b_{1},\ldots,b_{n}\rangle\coprod
             L\langle u_{\sigma}\mid\sigma\in MF(K)\rangle)/I}\).
\end{corollary}

\begin{proof}
Let $\widetilde{S}$ be the composite
\[\widetilde{S}\colon\nameddright
      {\bigvee_{\widetilde{\alpha}\in\widetilde{\mathcal{I}}} S^{t_{\widetilde{\alpha}}+1}}
      {E}{\Omega(\bigvee_{\widetilde{\alpha}\in\widetilde{\mathcal{I}}}
      S^{t_{\widetilde{\alpha}}+1})}{\Omega\widetilde{W}}{\Omega\djk}.\]
The definition of $\widetilde{W}$ implies that $\widetilde{S}_{\ast}$ induces
the composite
\[\widetilde{R}\hookrightarrow\lllnamedddright{UL\langle\widetilde{R}\rangle}
   {UL(i_{R})}{UL\langle R\rangle}{(\Omega W)_{\ast}}{\hlgy{\Omega\djkstwo}}
   {(\Omega\djk(\imath))_{\ast}}{\hlgy{\Omega\djk}}.\]
Now argue as in Corollary~\ref{Hurewiczzks}, using the description
of $(\Omega\djk(\imath))_{\ast}$ in Theorem~\ref{hlgyloopdjk2}, to obtain
the asserted inclusion in homology.
\end{proof}

We finish by bringing \zk\ back into the picture.

\begin{theorem}
   \label{zkdecomp}
   The map
   \(\namedright{\bigvee_{\widetilde{\alpha}\in\mathcal{\widetilde{I}}}
        S^{t_{\widetilde{\alpha}}+1}}{\widetilde{W}}{\djk}\)
   lifts to \zk, and induces a homotopy equivalence
   \(\namedright{\bigvee_{\widetilde{\alpha}\in\mathcal{\widetilde{I}}}
      S^{t_{\widetilde{\alpha}}+1}}{}{\zk}\).
\end{theorem}

\begin{proof}
Argue as in Theorem~\ref{zksdecomp}, using Proposition~\ref{hlgyloopzk}
and Corollary~\ref{Hurewiczzk} in place of Proposition~\ref{hlgyloopzks}
and Corollary~\ref{Hurewiczzks} respectively.
\end{proof}

\noindent
\textit{Proof of Theorem~\ref{djkmaps}}:
This is simply a rephrasing of Theorem~\ref{zkdecomp}.
$\qqed$

\section{Examples}
\label{sec:examples}

First, we consider the example of the shifted complex~(\ref{4vexample}) 
which appearedin the Introduction. 

\begin{example}
Let $K$ be the following simplicial complex on $4$ vertices 
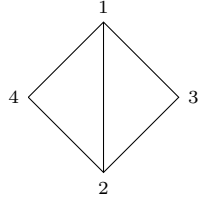
\[\begin{tikzpicture} 
   \draw (0,0)--(1,1)--(2,0)--(1,-1)--(0,0); 
   \draw (1,1)--(1,-1); 
   \node [above] at (1,1) {\scriptsize 1}; 
   \node [right] at (2,0) {\scriptsize 3}; 
   \node [below] at (1,-1) {\scriptsize 2}; 
   \node [left] at (0,0) {\scriptsize 4}; 
\end{tikzpicture}\] 
Under this ordering of the vertices, $K$ is shifted. The missing
faces of $K$ are given by $MF(K)=\big\{(3,4),(1,2,3),(1,2,4)\big\}$.
Observe that in this case $|K|=\bigcup_{\sigma\in
MF(K)}|\partial\sigma|$. In fact, 
$K=\partial (1,2,3)\cup\partial (1,2,4)$, where the boundaries 
of the two missing faces have been glued along the common face $(1,2)$. 
So $K$ is a directed $MF$-complex.  

By Theorem~\ref{ULcolim}, there is an algebra isomorphism
\[H_*(\Omega\djk(\underline{S}))\cong U\left(L_{ds}\langle
b_1,b_2,b_3,b_4\rangle\textstyle\coprod L\langle
u_1,u_2,u_3\rangle\right)/J\] 
where $u_{1}$ is the Hurewicz image of the adjoint of the Whitehead 
product corresponding to the missing face $(3,4)$, and $u_{2},u_{3}$ 
are the Hurewicz images of the adjoints of the higher Whitehead products 
corresponding to the missing faces$(1,2,3),(1,2,4)$, respectively. 
The ideal $J$ is determined by Jacobi identities and face relations 
based on the one missing face $(3,4)$ of dimension~$1$. Specifically, 
observe that the Jacobi identity gives 
$[u_{1},b_{1}]=[[b_{3},b_{4}],b_{1}]=[b_{3},[b_{4},b_{1}]]+
    [b_{4},[b_{3},b_{1}]]$.
As both $(1,3)$ and $(1,4)$ are faces of $K$, we have
$[b_{3},b_{1}]=0$ and $[b_{4},b_{1}]=0$. Therefore
$[u_{1},b_{1}]=0$. Similarly, $[u_{1},b_{2}]=0$. Thus
$J=([u_{1},b_{1}],[u_{1},b_{2}])$. 

By Theorem~\ref{djksmaps}, the wedge summands of \zks\ and 
the maps to \djks\ are as follows. For simplicity, we assume that 
each of the spheres in $\underline{S}$ is $S^{2}$. Part~(a) gives the 
three summands $S^{3}$, $S^{5}$ and $S^{5}$ with maps $w_{1}$, 
$w_{2}$ and $w_{3}$ respectively. Since the missing faces $(1,2,3)$ 
and $(1,2,4)$ are of dimension greater than~$1$, part~(b) gives the 
following additional summands and maps:  
\begin{equation} 
  \label{Wlist2} 
  [[w_{2},a_{j_{1}}],\ldots,a_{j_{l}}]\colon\namedright{S^{5+l}}{}{\djks} 
\end{equation}  
\begin{equation} 
  \label{Wlist3} 
  [[w_{3},a_{j_{1}}],\ldots,a_{j_{l}}]\colon\namedright{S^{5+l}}{}{\djks} 
\end{equation}  
for each list $1\leq j_{1}\leq\cdots\leq j_{l}\leq l$ with $1\leq l<\infty$. 
Note that the collection of spheres in~(\ref{Wlist2}) is identical 
to the collection in~(\ref{Wlist3}). In either case, let $W_{2}$ be the wedge 
of spheres. Since the missing face $(3,4)$ is of dimension~$1$, part~(c) 
gives the following additional summands:   
\begin{equation} 
  \label{Wlist1} 
  [[w_{1},a_{j_{1}}],\ldots,a_{j_{l}}]\colon\namedright{S^{3+l}}{}{\djks} 
\end{equation}  
for each list $1\leq j_{1}\leq\cdots\leq j_{l}\leq l$ with $1\leq l<\infty$ 
and each $j_{i}\in\{3,4\}$. Note that the restriction to $j_{i}\in\{3,4\}$ 
is from the fact that the Whitehead products $[w_{1},a_{1}],[w_{1},a_{2}]$ 
correspond to elements in the ideal $J$, so any iterated Whitehead product 
involving $a_{1}$ or $a_{2}$ must be excluded from the list of independent 
Whitehead products $W_{(3,4)}$. Let $W_{1}$ be the wedge of all possible 
spheres in~(\ref{Wlist1}). 

Collectively, we obtain a homotopy equivalence 
\[\zks\simeq S^{3}\vee 2 S^{5}\vee W_{1}\vee 2W_{2}\] 
and a map 
\[\namedright{\zks\simeq S^{3}\vee 2 S^{5}\vee W_{1}\vee 2W_{2}}{}{\djks}\] 
which is the wedge sum of $w_{1}$, $w_{2}$, $w_{3}$ and the 
three lists of iterated Whitehead products in~(\ref{Wlist1}), (\ref{Wlist2}) 
and~(\ref{Wlist3}). 

It is useful to reorganize the wedge summands of \zks. In general, the 
\emph{join} of two spaces $A$ and $B$ is denoted $A\ast B$; for our 
purposes it suffices to know that $A\ast B\simeq\Sigma A\wedge B$. 
The \emph{right half-smash} of $A$ and $B$ is the quotient space 
$A\rtimes B=(A\times B)/(\ast\times B)$. If $A$ is a suspension, then 
$A\rtimes B\simeq A\vee (A\wedge B)$. 
For $1\leq i\leq 4$, let $S^{2}_{i}$ be the copy of $S^{2}$ associated with 
vertex $i$. Observe that the wedge summands of $S^{3}\vee W_{1}$ are 
in one-to-one correspondence with the wedge summands of 
$\Omega S^{2}_{3}\ast\Omega S^{2}_{4}$, where the James splitting is 
used to iteratively decompose this space into a wedge of spheres. As well, 
the wedge summands of $S^{5}\vee W_{2}$ from~(\ref{Wlist2}) are in 
one-to-one correspondence with the wedge summands of 
$(\Omega S^{2}_{1}\ast\Omega S^{2}_{2}\ast\Omega S^{3}_{3})\rtimes\Omega S^{4}$, 
and the wedge summands of $S^{5}\vee W_{2}$ from~(\ref{Wlist3}) are in 
one-to-one correspondence with the wedge summands of 
$(\Omega S^{2}_{1}\ast\Omega S^{2}_{2}\ast\Omega S^{3}_{4})\rtimes\Omega S^{3}$.  
Thus the wedge decomposition of \zks\ above agrees with that 
in~\cite[Section 6]{GT2}. 

Next, by Theorem~\ref{hlgyloopdjk}, there is an algebra isomorphism
\[H_*(\Omega\djk)\cong U\left(L_{ds}\langle
b_1,b_2,b_3,b_4\rangle\textstyle{\coprod} L\langle
u_1,u_2,u_3\rangle\right)/(I+J)\]
where $u_1$ is the Hurewicz image of the adjoint of the
Whitehead product $\widetilde{w}_1\colon S^3\longrightarrow
\CP_3\vee\CP_4\longrightarrow \djk$, while $u_2$ and $u_3$ are the
Hurewicz images of the adjoints of the higher Whitehead products
$\widetilde{w}_2\colon S^5\longrightarrow FW(1,2,3)\longrightarrow\djk$, and
$\widetilde{w}_3\colon S^5\lra FW (1,2,4)\lra\djk$, respectively; $J$ is
as above, and $I=(b_i^2, [u_1,b_3], [u_2,b_4], [u_2,b_1],
[u_2,b_2], [u_2,b_3],[u_3,b_1],[u_3,b_2],[u_3,b_4])$.

By Theorem~\ref{djkmaps}, the wedge summands of \zk\ and the
maps to \djk\ are as follows. Part~(a) gives the three
summands $S^{3}$, $S^{5}$ and $S^{5}$ with maps $\widetilde{w}_{1}$,
$\widetilde{w}_{2}$ and $\widetilde{w}_{3}$ respectively. Since 
the missing faces $(1,2,3)$ and $(1,2,4)$ are of dimension 
greater than~$1$, part~(b) gives two additional summands $S^{6}$ 
and $S^{6}$ from the iterated Whitehead products
\[[\widetilde{w}_2, \tilde{a}_4]\colon S^6\lra\djk \text{ and }
   [\widetilde{w}_3, \tilde{a}_3]\colon S^6\lra\djk.\] 
For the missing face $(3,4)$ of dimension~$1$, part~(c) is 
vacuous in this case. To see this, observe that the Whitehead products
$[\widetilde{w}_{1},\tilde{a}_{3}]$ and $[\widetilde{w}_{1},\tilde{a}_{4}]$ 
correspond to the algebraic elements $[u_{1},a_{3}]$ and $[u_{1},a_{4}]$, 
both of which appear in the ideal $I$, while the Whitehead products 
$[\widetilde{w}_{1},\tilde{a}_{1}]$ and $[\widetilde{w}_{1},a_{2}]$  
correspond to the algebraic elements $[u_{1},b_{1}]$ and $[u_{1},b_{2}]$, 
both of which appear in the ideal $J$. Thus the collection of 
independent Whitehead products $\widetilde{W}_{(3,4)}$ is empty. 

Collectively then, we obtain a homotopy equivalence
$\zk\simeq S^{3}\vee 2S^{5}\vee 2S^{6}$ and a map
\[S^3\vee 2S^5\vee 2S^6 \longrightarrow\djk\]
which is the wedge sum of $\widetilde{w}_{1}$, $\widetilde{w}_{2}$,
$\widetilde{w}_{3}$, $[\widetilde{w}_{2},\tilde{a}_{4}]$ and
$[\widetilde{w}_{3},\tilde{a}_{3}]$. Note that the homotopy
equivalence matches that of~\cite[Example 10.2]{GT1}, which was
calculated by different means.
\end{example}

The next example is similar to the previous one, but boosted up 
one dimension. 

\begin{example} 
Let $K$ be the following simplicial complex on $5$ vertices 
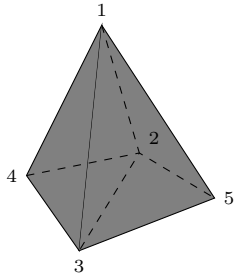
\[\begin{tikzpicture} 
   \draw [fill=gray] (0,0)--(1,2)--(0.7,-1)--(0,0); 
   \draw [fill=gray] (1,2)--(2.5,-0.3)--(0.7,-1); 
   \draw [dashed] (1,2)--(1.5,0.3)--(0.7,-1); 
   \draw [dashed] (0,0)--(1.5,0.3)--(2.5,-0.3); 
   \node [above] at (1,2) {\scriptsize 1}; 
   \node [above right] at (1.5,0.3) {\scriptsize 2}; 
   \node [below] at (0.7,-1) {\scriptsize 3}; 
   \node [left] at (0,0) {\scriptsize 4}; 
   \node [right] at (2.5,-0.3) {\scriptsize 5}; 
\end{tikzpicture}\] 
Here, $K=\partial(1,2,3,4)\cup\partial(1,2,3,5)$ where $\partial(1,2,3,4)$ 
and $\partial(1,2,3,5)$ have been glued along the common face $(1,2,3)$. 
This implies that $K$ is a directed $MF$-complex. Also, under this ordering 
of the vertices, $K$ is shifted. The missing faces of $K$ are given by 
$MF(K)=\{(4,5),(1,2,3,4),(1,2,3,5)\}$. This 
example is analogous to the previous one, so the algebraic descriptions 
of $\hlgy{\Omega\djks}$ and $\hlgy{\Omega\djk}$ are as before, 
but with a dimensional shift to account for the fact that the higher 
Whitehead products corresponding to the missing faces $(1,2,3,4)$ 
and $(1,2,3,5)$ are maps 
\(w_{i}\colon\namedright{S^{7}}{}{\djks}\) 
and 
\(\widetilde{w}_{i}\colon\namedright{S^{7}}{}{\djk}\)  
for $i\in\{2,3\}$. In particular, arguing as before, we obtain a 
homotopy equivalence $\zk\simeq S^{3}\vee 2S^{7}\vee 2S^{8}$ 
and a map 
\[\namedright{S^{3}\vee 2S^{7}\vee 2S^{8}}{}{\djk}\] 
which is the wedge sum of $\widetilde{w}_{1}$, $\widetilde{w}_{2}$,
$\widetilde{w}_{3}$, $[\widetilde{w}_{2},\tilde{a}_{5}]$ and
$[\widetilde{w}_{3},\tilde{a}_{4}]$.
\end{example}

\bibliographystyle{amsalpha}

\end{document}